\documentclass[draft,11pt]{amsart}
\usepackage{amsmath}
\usepackage{amssymb}
\usepackage{amscd}
\usepackage{auto-pst-pdf}

\usepackage{url}
\usepackage[all]{xypic}

\usepackage{pstricks}

\usepackage{graphicx}

\newcommand{\intfrac}[2]{\genfrac{\lfloor}{\rfloor}{}{}{#1}{#2}}

\newcommand{\sawtooth}[1]{\left\langle #1\right\rangle}
\newcommand{\bp}{\begin{pmatrix}}
\newcommand{\ep}{\end{pmatrix}}
\newcommand{\be}{\begin{equation}}
\newcommand{\ee}{\end{equation}}

\numberwithin{equation}{section}

\theoremstyle{plain}
\newtheorem{theorem}[equation]{Theorem}
\newtheorem{lemma}[equation]{Lemma}
\newtheorem{proposition}[equation]{Proposition}

\newtheorem{corollary}[equation]{Corollary}

\theoremstyle{definition}
\newtheorem{example}[equation]{Example}
\newtheorem{remark}[equation]{Remark}

\newtheorem{definition}[equation]{Definition}

\newtheorem{ex}[equation]{Example}

\numberwithin{equation}{section}

\def\Z{\mathbb Z}

\def\Q{\mathbb Q}
\def\C{\mathbb C}

\def\p{\partial}

\def\sm{\setminus}

\def\spinc{Spin$^c$}
\newcommand{\hide}[1]{}

\newrgbcolor{lightgreen}{0.95 1.0 0.95}
\DeclareMathOperator\coker{coker}

\DeclareMathOperator\Sp{Sp}
\DeclareMathOperator\Hom{Hom}
\DeclareMathOperator\codim{codim}
\newrgbcolor{lightgreen}{0.95 1.0 0.95}
\newrgbcolor{lightgreenbutnotsolight}{0.85 1.0 0.85}
\newrgbcolor{darkgreen}{0.2 0.6 0.2}
\newrgbcolor{lightblue}{0.8 0.8 1.0}
\newrgbcolor{moderateblue}{0.5 0.5 1.0}
\newrgbcolor{orange}{1.0 0.7 0.0}

\def\Z{\mathbb Z}

\def\Q{\mathbb Q}
\def\C{\mathbb C}

\def\p{\partial}

\def\sm{\setminus}

\def\spinc{Spin$^c$}
\newcommand{\sss}{\ifmmode{{\mathfrak s}}\else{${\mathfrak s}$\ }\fi}
\newcommand{\sst}{\ifmmode{{\mathfrak t}}\else{${\mathfrak t}$\ }\fi}

\DeclareMathOperator\ord{ord}
\DeclareMathOperator\lk{lk}

\newrgbcolor{lightgreen}{0.95 1.0 0.95}
\newrgbcolor{lightgreen}{0.95 1.0 0.95}
\newrgbcolor{lightgreenbutnotsolight}{0.85 1.0 0.85}
\newrgbcolor{darkgreen}{0.2 0.6 0.2}
\newrgbcolor{lightblue}{0.8 0.8 1.0}
\newrgbcolor{moderateblue}{0.5 0.5 1.0}
\newrgbcolor{orange}{1.0 0.7 0.0}
\def\necklaceshift{1.5}
\def\necklaceradius{1}
\def\labelshift{1.2}
\def\crosslinkstart{0.8}
\def\crosslinkend{1.2}
\def\secondlabelshift{1.8}

\def\necklacetemplate#1{\psset{linewidth=1.5pt}\rput(0,\labelshift){\ensuremath{#1}}}
\def\necklaceright#1{\necklacetemplate{#1}\psarc(0,0){\necklaceradius}{350}{290}}
\def\necklaceleft#1{\necklacetemplate{#1}\psarc(0,0){\necklaceradius}{180}{120}}

\def\crosslink#1{\rput{315}(0,0){\psline[linewidth=1pt](\crosslinkstart,0)(\crosslinkend,0)\rput(\secondlabelshift,0){\rotatebox{45}{#1}}}}
\def\othercrosslink#1{\rput{225}(0,0){\psline[linewidth=1pt](\crosslinkstart,0)(\crosslinkend,0)\rput(\secondlabelshift,0){\rotatebox{135}{#1}}}}
\def\cnecklaceleft#1#2{\necklacetemplate{#1}\psarc(0,0){\necklaceradius}{180}{120}\crosslink{#2}}
\def\cnecklaceright#1#2{\necklaceright{#1}\othercrosslink{#2}}

\def\necklacemiddle#1{\necklacetemplate{#1}\psarc(0,0){\necklaceradius}{180}{290}\psarc(0,0){\necklaceradius}{350}{120}}
\def\templatedoublenecklace#1#2#3#4{\psset{linecolor=#3}\rput(0,0){\necklaceright{#1}}\psset{linecolor=#4}\rput(\necklaceshift,0){\necklaceleft{#2}}}
\def\templatetriplenecklace#1#2#3#4#5#6{\psset{linecolor=#4}\rput(-\necklaceshift,0){\necklaceright{#1}}\psset{linecolor=#5}\rput(0,0){\necklacemiddle{#2}}%
\psset{linecolor=#6}\rput(\necklaceshift,0){\necklaceleft{#3}}}

\def\blackdoublenecklace#1#2{\templatedoublenecklace{#1}{#2}{black}{black}}

\def\colortriplenecklace#1#2#3{\templatetriplenecklace{#1}{#2}{#3}{blue}{green}{orange}}

\def\contractmiddlenecklace#1#2{\templatedoublenecklace{#1}{#2}{blue}{orange}}

\def\contractleftnecklace#1#2{\templatedoublenecklace{#1}{#2}{blue}{green}}

\def\templatedoublecnecklace#1#2#3#4#5#6{\psset{linecolor=#3}\rput(0,0){\cnecklaceright{#1}{#5}}\psset{linecolor=#4}\rput(\necklaceshift,0){\cnecklaceleft{#2}{#6}}}
\def\templatetriplecnecklace#1#2#3#4#5#6#7#8{\psset{linecolor=#4}\rput(-\necklaceshift,0){\cnecklaceright{#1}{#7}}%
\psset{linecolor=#5}\rput(0,0){\necklacemiddle{#2}}%
\psset{linecolor=#6}\rput(\necklaceshift,0){\cnecklaceleft{#3}{#8}}}
\def\blackdoublecnecklace#1#2#3#4{\templatedoublecnecklace{#1}{#2}{black}{black}{#3}{#4}}

\def\colortriplecnecklace#1#2#3#4#5{\templatetriplecnecklace{#1}{#2}{#3}{blue}{green}{orange}{#4}{#5}}

\def\contractmiddlecnecklace#1#2#3#4{\templatedoublecnecklace{#1}{#2}{blue}{orange}{#3}{#4}}

\def\contractleftcnecklace#1#2#3#4{\templatedoublecnecklace{#1}{#2}{blue}{green}{#3}{#4}}

\def\nodedot#1{\pscircle[linecolor=black,fillcolor=black,fillstyle=solid]#1{0.08}}

\title[Rational cuspidal curves in Hirzebruch surfaces]{Topological obstructions for rational cuspidal curves in Hirzebruch surfaces}
\author{Maciej Borodzik}
\address{Institute of Mathematics, University of Warsaw, ul. Banacha 2,
02-097 Warsaw, Poland}
\email{mcboro@mimuw.edu.pl}
\author{Torgunn Karoline Moe}
\address{Department of Mathematics, University of Oslo, Oslo, Norway}
\email{t.k.moe@math.uio.no}

\date{\today}
\makeatletter
\@namedef{subjclassname@2010}{\textup{2010} Mathematics Subject Classification}
\makeatother

\begin{document}
\begin{abstract}
We study rational cuspidal curves in Hirzebruch surfaces. We provide two obstructions for the existence of rational cuspidal curves
in Hirzebruch surfaces with prescribed types of singular points. The first result comes from Heegaard--Floer theory and is a generalization of a result
by Livingston and the first author. The second criterion is obtained by comparing the spectrum of a suitably defined link at infinity of a curve with
spectra of its singular points.
\end{abstract}
\maketitle

\section{Introduction}
Let $C$ be a reduced and irreducible algebraic curve in a smooth complex surface $X$. A singular point $p$ on $C$ is called a \emph{cusp} if it is locally irreducible. The curve is called \emph{cuspidal} if all its singularities are cusps.

Cuspidal curves in the projective plane have been investigated in classical algebraic geometry and have been subject of intense study the past three decades. The renewed interest in these curves in the eighties came after results by Lin and Zaidenberg in \cite{LinZaidenberg}, and Matsuoka and Sakai in \cite{MatsuokaSakai}. Moreover, two questions about plane cuspidal curves were asked by Sakai in 1994 (see \cite{Sakai}), and ever since, several attempts have been made to describe and classify rational cuspidal curves in the projective plane (see \cite{FLMN04, Borodzik, Fent, FlZa95, FlZa97, Liu, PALKA, Piontkowski, Tono05, Wak}).

In \cite{MOEPHD} the second author turned the attention to cuspidal curves in Hirzebruch surfaces and found that many of the results for plane cuspidal curves could be extended to curves in Hirzebruch surfaces (see \cite{MOEONC, MOECCH}). Indeed, this does not come as a surprise, since the Hirzebruch surfaces are linked to each other and the projective plane by birational transformations, and such transformations clearly transform rational curves to rational curves. However, the picture is somewhat more complicated; in general, a cuspidal curve might acquire some multibranched singular points under a birational transformation. Therefore, there is no direct
correspondence between rational cuspidal curves in $\mathbb{C}P^2$ and rational cuspidal curves in Hirzebruch surfaces.

In the present article we continue this work and extend two results from the plane case to the case of cuspidal curves in Hirzebruch surfaces. The first result, given in Theorem~\ref{THF}, is a consequence of Heegaard--Floer theory, and it is a generalization of the result by Livingston and the first author in \cite{BorodzikHF}.  We refer to Section~\ref{NOT}
for explaining notation used in the theorem and especially to Section~\ref{sec:semigroups} for the definition of the function $R$.

\begin{theorem}\label{THF}
Let $C$ be a rational cuspidal curve of type $(a,b)$ in a Hirzebruch surface $X_e$ with $e \geq 0$. Let $g=(a-1)(b-1)+\frac12b(b-1)e$.
Then for any $m\in[-g,g]$ and for any presentation $m+g=s_1b+s_2(a+be)+1$, where $s_1$ and $s_2$ are integers, we have
\begin{equation}\label{eq:THF}
R(m+g)\geq P(s_1,s_2),
\end{equation}
where $R$ is the counting function for the semigroups of the singular points of $C$ and
\[P(s_1,s_2)=(s_1+1)(s_2+1)+\frac12s_2(s_2+1)e.\]
\end{theorem}
Notice that if $m+g-1$ is not divisible by $\gcd(a,b)$, then Theorem~\ref{THF} does not provide a direct restriction on the value of $R(m+g)$.

In Section~\ref{ss:bezout} we show an alternative, algebraic proof of inequality \eqref{eq:THF}, following the ideas of \cite{FLMN04}. 
As a matter of fact, Theorem~\ref{thm:generalsurface} gives a lower bound for a function $R$ for any rational cuspidal curve in any algebraic surface.
It is natural to conjecture that the $d$--invariants estimate used in the proof of Theorem~\ref{THF} will give the same bound.

We also remark, that if $m+g\not\in[-g,g]$, the value of $R(m+g)$ is fixed: it is $0$ if $m+g<0$ and $m$ if $m+g>2g$.

Our second result is about the semicontinuity property of the spectrum. It puts restrictions on the spectrum of singular points of a rational cuspidal curve
in $X_e$.

\begin{theorem}\label{thm:spectrum}
Let $C$ be a rational cuspidal curve of type $(a,b)$ in $X_e$. Let $Sp_1,\ldots,Sp_n$ be the spectra of its singular points. Let $Sp^\infty_{a,b}$ be
the spectrum at infinity given in Table~\ref{table:one}. Then for any $x\in(0,1)$ such that $x\not\in Sp^\infty_{a,b}$ we have
\begin{equation}\label{eq:semic}
\begin{split}
\sum_{j=1}^n\#Sp_j\cap(x,x+1)&\le \#Sp^\infty_{a,b}\cap (x,x+1)\\
\sum_{j=1}^n\#Sp_j\setminus(x,x+1)&\le \#Sp^\infty_{a,b}\setminus (x,x+1).
\end{split}
\end{equation}
\end{theorem}
\begin{table}
\begin{tabular}{|l||l|l|}\hline
&$\vphantom{\Bigl\Vert}$\textbf{$x$ is of form }&\textbf{multiplicity of $x$ in $Sp^\infty_{a,b}$}\\\hline
1.&$\vphantom{\Bigl\Vert}$ $x=1$ & $a+b-1$\\\hline
2.&$\vphantom{\Bigl\Vert}$ $x=\frac{p}{w}$ and $x=\frac{q}{b}$ for some $p$ and $q$ & $\intfrac{pb}{w}+\intfrac{qa}{b}-1$\\\hline
3.&$\vphantom{\Bigl\Vert}$ $x=1+\frac{p}{w}$ and $x=1+\frac{q}{b}$ for some $p$ and $q$ & $a+b-1-\intfrac{pb}{w}-\intfrac{qa}{b}$\\\hline
4.&$\vphantom{\Bigl\Vert}$ $x=\frac{p}{w}$ for some $p$ but $x\neq \frac{q}{b}$ for any $q$ & $\intfrac{pb}{w}$\\\hline
5.&$\vphantom{\Bigl\Vert}$ $x=1+\frac{p}{w}$ for some $p$ but $x\neq 1+\frac{q}{b}$ for any $q$ & $b-1-\intfrac{pb}{w}$\\\hline
6.&$\vphantom{\Bigl\Vert}$ $x=\frac{q}{b}$ for some $q$ but $x\neq \frac{p}{w}$ for any $p$ & $\intfrac{qa}{b}$\\\hline
7.&$\vphantom{\Bigl\Vert}$ $x=1+\frac{q}{b}$ for some $q$ but $x\neq 1+\frac{p}{w}$ for any $p$ & $a-1-\intfrac{qa}{b}$\\\hline
8.&$\vphantom{\Bigl\Vert}$ For all other $x$  & 0 \\ \hline
\end{tabular}
\medskip
\caption{Spectrum at infinity of a type $(a,b)$ curve. Here $w=a+be$, and $p,q$ are assumed to be integers, $1\le p\le w-1$, $1\le q\le b-1$. 
Number $x$ is in the interval $[0,2]$.}\label{table:one}
\end{table}

The two above results, Theorem~\ref{THF} and Theorem~\ref{thm:spectrum}, give two restrictions for possible configurations of singular points
on a cuspidal curve. As we show in Section~\ref{Examples}, the two results differ in natures. Indeed, for unicuspidal curves, the semigroup distribution
property obstructs cases where the multiplicities are large (close to $b$), while the spectrum semicontinuity is effective in obstructing curves with low 
multiplicities.
\subsection{Structure}
In this article we first set up the notation in Section~\ref{NOT}. In Section~\ref{HF} we use Heegaard--Floer theory to establish Theorem~\ref{THF}. In Section~\ref{Spectra} we study the link at infinity of curves in Hirzebruch surfaces. In Section~\ref{S} we show how the spectrum of the link at infinity can be computed, and the result is as shown in Table~\ref{table:one}. Our main result in this section, giving new restrictions for cuspidal curves, is Theorem~\ref{thm:spectrum} that compares the spectrum of singular points of the curve to the spectrum of the link at infinity.
Finally, in Section~\ref{Examples} we give some examples of possible applications.

\section{Generalities}\label{NOT}
\subsection{Hirzebruch surfaces}
Let $X_e$, $e\geq 0$, be a Hirzebruch surface, regarded as a projectivisation of a rank 2 bundle $\mathcal{O}\oplus\mathcal{O}(-e)$ over $\C P^1$. Let $L$
be a fibre and $M_0$ the special section, so that the intersections and self--intersections are as follows:
\[L^2=0,\ M_0^2=-e,\ L\cdot M_0=1.\] 
We define $M=eL+M_0$. Then $L$ and $M$ generate
$H_2(X_e;\Z)$ and the intersection matrix is
\[\bp 0 & 1 \\ 1 & e \ep.\]
\begin{definition}
For integers $a \geq 0$ and $b>0$ ($b \geq 0$ when $e=0$), a curve $C\subset X_e$ is \emph{of type $(a,b)$} if it is irreducible and
its homology class is $aL+bM\in H_2(X_e)$. 
\end{definition}
\begin{remark}
Unless stated otherwise, we shall suppose that $C$ is 
rational and cuspidal.
\end{remark}

We denote 
\[d=C^2=(aL+bM)^2=2ab+b^2e.\]
Furthermore let
\[c=\gcd(a,b),\ a'=a/c,\ b'=b/c.\]
The arithmetic genus of $C$ is given by the formula (see \cite[Corollary 3.1.4]{MOEPHD}). 
\begin{equation}\label{eq:genus}
g=(a-1)(b-1)+\frac12b(b-1)e.
\end{equation}

\subsection{Singular points and semigroups}\label{sec:semigroups}

(We refer to \cite[Chapter 4]{Wall-book} for more details about semigroups of singular points.)

Let $z$ be a cuspidal singular point of $C$. We can associate with $z$ a semigroup $S_z$ of non--negative integers. For a quasi--homogeneous
singularity given by $x^p-y^q=0$ with $p,q$ coprime, the semigroup is generated by $p$ and $q$. We always assume that $0\in S_z$.

Given any semigroup in $\mathbb{Z}_{\ge 0}$ we define a function $R_S\colon \mathbb{Z}\to\mathbb{Z}$
\begin{equation}\label{eq:defR} R_S(t)=\# S\cap [0,t).
\end{equation}
We have the following fact. 
\begin{lemma}\label{lem:standardvaluesofR}
If $S$ is a semigroup of a singular point $z$ with Milnor number $\mu$ (the genus of the link of the singular point is then $\mu/2$),
then for all $m\ge 0$ we have $R_S(m+\mu)=m+\mu/2$.
\end{lemma}
\begin{proof}
The complement $\Z_{\ge 0}\setminus S$ has precisely $\mu/2$ elements and the largest is $\mu-1$; see \cite[Chapter 4]{Wall-book}.
\end{proof}

Given any two functions $R_1,R_2\colon\Z\to\Z$ bounded from below we define their \emph{infimum convolution} to be
\[R_1\diamond R_2(t)=\min_{k\in\Z} R_1(k)+R_2(t-k).\]
The infimum convolution is clearly commutative and associative.

\begin{definition}\label{def:Rfunction}
Let $C\subset X_e$ be a cuspidal curve (not necessarily rational). Then the $R$--function of $C$ is defined as
\[R=R_{S_1}\diamond R_{S_2}\diamond \ldots\diamond R_{S_n},\]
where $S_1,\ldots,S_n$ are semigroups corresponding to singular points of $C$.
\end{definition}
We have the following corollary to Lemma~\ref{lem:standardvaluesofR}.
\begin{corollary}\label{cor:standardvaluesofR2}
If $g$ is the sum of genera of the links of singular points of a cuspidal curve $C$, then $R(2g+m)=g+m$ for any $m\ge 0$.
\end{corollary}

\section{A criterion from the complement of $C$ and related invariants}\label{HF}
Suppose that $C$ is a rational cuspidal curve in $X_e$. We consider $N$, a tubular neighbourhood of $C$ in $X_e$. Let $Y$ be the boundary of $N$ with \emph{reversed}
orientation and let $W=X_e\sm N$. We have
\[\partial W=Y.\]

The main goal of this section is to give a proof of Theorem~\ref{THF}. The key result in our proof is the following theorem by Ozsv\'ath and Szab\'o; see 
\cite[Theorem 9.6]{OS-absolutely}. It relies on the fact that the intersection form $H_2(W)\times H_2(W)\to\Z$ is negative definite.

\begin{theorem}\label{thm:os}
For any \spinc{} structure $\sss$ on $Y$ extending to a \spinc{} structure $\sst$ on $W$ we have
\begin{equation}\label{eq:osd} d(Y,\sss)\ge \frac14\left( c_1^2(\sst)-3\sigma(W)-2\chi(W) \right),\end{equation}
where $d(Y,\sss)$ is a correction term.
\end{theorem}
In order to use this result we need to decrypt the information encoded in inequality \eqref{eq:osd}. We will do this in the following steps, following
the pattern used in \cite{BCG14, BHL14,BorodzikHF}.
\begin{itemize}
\item Describe $Y$ as a surgery on a knot in $S^3$ and compute its $d$--invariants.
\item Study homological properties of $W$, in particular show that the intersection form is negative definite.
\item Check which \spinc{} structures on $Y$ extend over $W$.
\item Compute $c_1^2(\sst)$ for such structures.
\item Compute $d(Y,\sss)$.
\end{itemize}
\subsection{The manifold $Y$ and its $d$--invariants}\label{sec:onY}

We shall need the following characterization of $Y$.
\begin{proposition}\label{prop:surgerycoeff}
Let $K_1,\ldots,K_n$ be the links of the singularities on the curve $C$, and $K=K_1\#\ldots\#K_n$. 
Then $-Y$ is a surgery on $S^3$ along $K$ with surgery coefficient $C^2=2ab+b^2e=d$.
\end{proposition}
\begin{proof}
The proof is the same as in the case when $X_e$ is the projective plane, see \cite{BorodzikHF}. 
\end{proof}
As a corollary we can write down homologies of $Y$.
\begin{corollary}
We have $H_1(Y)=\Z_d$ and $H_2(Y)=0$.
\end{corollary}
\begin{lemma}\label{lem:gK}
The genus of the knot $K$ is equal to $g$ (defined in \eqref{eq:genus}).
\end{lemma}
\begin{proof}
This follows immediately from the genus formula \eqref{eq:genus} and the fact that $C$ is rational.
\end{proof}
\begin{remark}\label{rem:d2g}
We notice that  $d-2g=2a+2b+be-2>0$.
\end{remark}

Since $Y$ is presented as an integer surgery on a knot with slope $d$, it has an enumeration of \spinc{} structures $\sss_m$, where
$m\in[-d/2,d/2)$. Details are presented in Section~\ref{sec:spinc}. Given that result, we have the following:

\begin{proposition}[see \expandafter{\cite[Theorem 5.1]{BorodzikHF}}]\label{prop:dinv}
The $d$--invariant $d(Y,\sss_m)$ is equal to
\[-d(Y,\sss_m)=\frac{(d-2m)^2-d}{4d}-2(R(m+g)-m),\]
where $R$ is the $R$--function from Definition~\ref{def:Rfunction}. 
\end{proposition}
\begin{proof}
The proof of Proposition~\ref{prop:dinv} consists of two steps. First, the knot $K$ (see Proposition~\ref{prop:surgerycoeff})
is a connected sum of algebraic knots (or, more generally, $L$--space knots). Therefore one can compute the Heegaard--Floer chain complex $CFK^\infty(K)$
using the Alexander polynomial of $K_1,\ldots,K_n$. The Alexander polynomial of an algebraic knot is tightly related to the semigroup of the singular point
(see \cite{Wall-book}), so the $R$--function enters the formula.

The second part is expressing Heegaard--Floer homology of $+d$ surgery on $K$ (that is of $-Y$) in terms of $CFK^\infty(K)$. This part uses the fact that
$d>2g$; see Remark~\ref{rem:d2g} above. Lastly, we note that by \cite{OS-absolutely} reversing the orientation amounts to reversing the sign of the $d$--invariant.
\end{proof}

\subsection{Homological properties of $W$}

We begin with the following result.
\begin{lemma}
We have $H_2(W)=\Z$ and $H_1(W)=\Z/c\Z$.
\end{lemma}
\begin{proof}
Set $X=X_e$. Consider the long exact sequence of the pair $(X,W)$. 
By excision, $H_*(X,W)\cong H_*(N,Y)$. The latter group is $\Z$ in degrees $2,4$ and $0$ otherwise by Thom isomorphism. Hence $H_3(W)=\Z$ and the long
exact sequence of the pair gives two following exact sequences.
\[\xymatrix@R=0.01in @C=0.25in{
0\ar[r]& H_4(X)\cong\Z\ar[r]&H_4(X,W)\cong\Z\ar[r]&H_3(W)\ar[r]&0&\\
0\ar[r]& H_2(W)\ar[r]&H_2(X)\ar[r]& H_2(X,W)\cong\Z\ar[r]&H_1(W)\ar[r]& 0.\\}\]
It follows that  $b_3(W)=0$. To study $H_2(W)$ we observe that
the map from $H_2(X)\to H_2(X,W)\cong \Z$ can be described explicitly. Namely, for $z\in H_2(X)$ represented by a (real) surface $Z$ intersecting
$C$ transversally, $z$ is mapped to $Z\cdot C$ times the generator. In particular, $L$ is mapped to $L\cdot C=b$ and $M$ is mapped to $M\cdot C=a+be$.
The image of $H_2(X)$ in $H_2(X,W)$ is therefore generated by $a,b$. Hence
\[H_2(W)=\Z\textrm{ and }H_1(W)=\Z/c\Z.\]
\end{proof}

Our aim is to compute the intersection form on $W$. Notice that the class $H=(a'+b'e)L-b'M$ intersects trivially with $C$,
so it belongs to the kernel of $H_2(X)\to H_2(X,W)$. In particular it descends to a class in $H_2(W)$ (by a slight abuse
of notation we will still denote it by $H$) with self--intersection equal to
\[(-2a'b'-b'b'e)=-\frac{d}{c^2}.\]
\begin{lemma}
The class $H$ generates $H_2(W)$.
\end{lemma}
\begin{proof}
The intersection form on $W$, that is the map $H_2(W)\to \Hom(H_2(W),\Z)$ is given by the sequence of maps
\[H_2(W)\to H_2(W,Y)\stackrel{\simeq}\to H^2(W)\stackrel{\simeq}{\leftarrow} \Hom(H_2(W),\Z).\]
Here the first map is the exact sequence of the pair $(W,Y)$, the second map is Poincar\'e duality isomorphism and the third follows from the universal
coefficient theorem. In particular, $\coker \left(H_2(W)\to \Hom(H_2(W),\Z)\right)$ is isomorphic to $\coker \left(H_2(W)\to H_2(W,Y)\right)$. With $Y=\partial W$, we have
\[0\to \coker\left( H_2(W)\to H_2(W,\partial W)\right)\to H_1(Y)\to H_1(W)\to 0.\]
All the groups in the short exact sequence are finite, so taking the cardinalities we obtain 
\[|H_1(Y)|=|H_1(W)|\cdot |\coker \left(H_2(W)\to H_2(W,\partial W)\right)|.\]
It follows that
\[|\coker \left(H_2(W)\to \Hom(H_2(W),\Z)\right)|=d/c^2.\]
As $H_2(W)$ has rank one, the cardinality of the cokernel is precisely the absolute value of the self--intersection of a generator. 
If $H$ were a nontrivial multiple of a generator, the cokernel of the intersection form would be smaller than $d/c^2$.
\end{proof}

We notice that by the universal coefficient theorem, $H^2(W)\cong \Z\oplus\Z/c\Z$, in particular $H^2(W;\Q)\cong\Q$. The classes $L$ and $M$
can be regarded as classes in $H^2(W)$, under the composition $H_2(X)\to H^2(X)\to H^2(W)$, where the first map is Poincar\'e duality and the
second is the restriction homomorphism. 
We have $L\cdot H=-b'$ and $M\cdot H=a'$, so in $H^2(W;\Q)$, $L$ is $-b'$ times the generator and $M$ is $a'$ times the generator.
Since the intersection form on $H^2(W)$ is the inverse of the intersection form on the non-torsion part of $H_2(W)$, the classes in $H^2(W;\Z)$ represented
by $L$ and $M$ have the following intersections.

\begin{equation}
L^2=-\frac{b^2}{d},\ M^2=-\frac{a^2}{d},\ L\cdot M=-\frac{ab}{d}.\label{eq:lsquare}
\end{equation}

\subsection{\spinc{} structures on $Y$ and $W$}\label{sec:spinc}
A \spinc{}  structure on a manifold $N$ is a choice of a complex line bundle $L$ over $N$ and of a spin structure on the bundle $TN\otimes L^{-1}$.
When speaking of restricting or prolonging \spinc{}  structures from one submanifold to another, the intuition that the \spinc{}  structures are `like
line bundles' is very convenient.

On a rational homology sphere, \spinc{}  structures are in a bijective correspondence with elements in $H_1(M,\Z)\cong H^2(M,\Z)$. If $M$
is represented as a integral surgery along a knot in $S^3$ we have a simple description of the \spinc{}  structures on $M$, see \cite[Section 4]{OS-knot}. 
\begin{proposition}
Let $K$ be a knot in $S^1$ and $d>0$ be an integer.
Let $N$ be a four manifold obtained by attaching a two--handle to a ball $B^4$ along $K$ with  framing $d$ (in this way $\partial N=S^3_d(K)=:M$).
Then
\begin{itemize}
\item Any \spinc{}  structure on $M$ extends to $N$.
\item For any $m\in[-d/2,d/2)$, there is a \emph{unique} \spinc{}  structure on $M$, which extends to a \spinc{}  structure $\sst_m$ over $N$
such that
\[\langle c_1(\sst_m),\Sigma\rangle +2m=d,\]
where $\Sigma$ is a generator of $H_2(N)$ consisting of the core of the two handle capped with a Seifert surface for $K$.
\end{itemize}
\end{proposition}

The above proposition allows us to characterise the \spinc{}  structures for $-Y$.

\begin{definition}
For any $m\in[-d/2,d/2)$ the \spinc{}  structure $\sss_m$ on $-Y$ is the \spinc{}  structure that extends to $\sst_m$ over $N$.
\end{definition}
Notice that a \spinc{} structure on $-Y$ induces a \spinc{} structure on $Y$.
We ask, which \spinc{} structures on $Y$ extend over $W$ and what is the first Chern class of such an extended \spinc{} structure. To answer
this question we note that if a \spinc{} structure $\sss_m$ on $Y$ extends over $W$, then it can be glued with $\sst_m$ on $N$ to form a \spinc{} structure
on the whole Hirzebruch surface $X_e$. Conversely, a \spinc{} structure on $X_e$ can be restricted to $W$. To study which \spinc{} structures on $Y$ extend
over $W$ it is enough to study restrictions of \spinc{} structures on $X_e$ to $W$.

By \cite[Section 1.4.2]{GS99}, \spinc{} structures on $X_e$ correspond to so--called characteristic elements on $X_e$, that is, to classes $x\in H^2(X_e)$ such that
for any $w\in H_2(X_e)$ we have $\langle x,w\rangle=w\cdot w\bmod 2$. Therefore, characteristic elements on $X_e$ are classes
$r_1L+r_2M$ such that $r_2$ is even and $r_1$ is congruent to $e$ modulo $2$.

So let us consider a class $r_1L+r_2M$ with $r_1,r_2$ as above. The corresponding \spinc{} structure on $X_e$ restricts to the \spinc{} structure $\sst$
on $N$ with 
\[\langle c_1(\sst),C\rangle=(r_1L+r_2M)\cdot C=r_1b+r_2a+r_2be.\] 
Let us define 
\[k=r_1b+r_2a+r_2be\textrm{ and }m=(d-k)/2.\] 
If $m\in[-d/2,d/2)$, it follows that $\sst$ restricts
to the class $\sss_m$ on $W$. 
Using \eqref{eq:lsquare} we can compute the square of the first Chern class of $\sst_m$
on $W$. Let us summarise this discussion in the following result.

\begin{lemma}\label{lem:spinc}
Let $m\in[-d/2,d/2)$. If $m=(d-k)/2$ is such that $k$ can be presented as $r_1b+r_2a+r_2be$ for $r_2$ even and $r_1\equiv e\bmod 2$, 
then the \spinc{} structure $\sss_m$ on $Y$
extends to a \spinc{} structure on $W$ having $c_1^2$ equal to
\[-\frac{(r_1b-r_2a)^2}{d}.\]
\end{lemma}
Notice that if $k$ can be presented in more than one way, it means that the extension of $\sss_m$ is not unique.

\subsection{Proof of Theorem \ref{THF}}
Let us now choose two integers $s_1$ and $s_2$. Set
\[r_1=2a+e-2-2s_1 \textrm{ and } r_2=2b-2-2s_2.\]
As above we write $k=r_1b+r_2(a+be)$ and $m=(d-k)/2$. We have
\[m+g=s_2b+s_1(a+be)+1.\]

Therefore, Lemma~\ref{lem:spinc} and Theorem~\ref{thm:os} together with computations of $d$--invariants in Proposition~\ref{prop:dinv} gives us

\begin{equation}\label{eq:rmg}
\frac{k^2-(r_1b-r_2a)^2}{8d}\le R(m+g)-m.
\end{equation}
After straightforward, but tedious computations, we obtain
\begin{equation}\label{eq:onRsecondtime}
R(s_1b+s_2(a+be)+1)\ge (s_1+1)(s_2+1)+\frac12s_2(s_2+1)e=:P(s_1,s_2).
\end{equation}

To conclude the proof we need to find the range for which \eqref{eq:onRsecondtime} holds. We have $m\in[-d/2,d/2)$, hence $m+g\in[g-d/2,g+d/2)$. By
Remark~\ref{rem:d2g} we have $d>2g$, so $g-d/2<0$ and $g+d/2>2g$. Thus $[0,2g]\subset[g-d/2,g+d/2]$.
The values of $R(k)$ for $k$ outside of $[0,2g]$ are well understood; see Corollary~\ref{cor:standardvaluesofR2}. 

\subsection{Alternative proof of Theorem~\ref{THF} via B\'ezout--like argument}\label{ss:bezout}
The following result is a direct generalisation of \cite[Proposition 2]{FLMN04}, and it provides an algebraic proof of Theorem~\ref{THF}.

\begin{theorem}\label{thm:generalsurface}
Let $X$ be a projective algebraic surface and $C$ a rational cuspidal curve on it with singular points $z_1,\ldots,z_n$. 
Let $\mathcal{L}$ be a line bundle on $X$ such that 
$t:=\langle c_1(\mathcal{L}),C\rangle>0$, the space of global sections $\Gamma(X,\mathcal{L})$ has positive dimension
and $\mathcal{L}$ has no section vanishing entirely on $C$. Then for the $R$--function as in Definition~\ref{def:Rfunction} we have
\[R(t+1)\ge \dim \Gamma(X,\mathcal{L}).\]
\end{theorem}
\begin{proof}
Denote $U=\Gamma(X,\mathcal{L})$. For a section $u\in U$ we denote by $m_j(u)$ the intersection multiplicity of $u^{-1}(0)$ and $C$ at $z_j$.
If $u$ does not vanish at $z_j$, we set $m_j(u)=0$.

Given non-negative integers $p_1,\ldots,p_n$, set 
\[V=V(p_1,\ldots,p_n):=\{u\in U\colon \forall_j\, m_j(u)\ge p_j\},\]
Then $V$ is a linear subspace of $U$. 

\begin{lemma}\label{lem:oncodim}
We have
\[\codim V\le \sum \#S_i\cap[0,p_j),\]
where $S_j$ is the semigroup associated with the singular point $z_j$.
\end{lemma}

Given this result, we finish the proof of Theorem~\ref{thm:generalsurface} by contradiction. Namely fix $p>t$ and suppose $p_1,\ldots,p_n$ satisfy
$p_1+\ldots+p_n=p$, but $\sum\#S_j\cap[0,p_j)<\dim U$. By Lemma~\ref{lem:oncodim} the space $V(p_1,\ldots,p_n)$ has positive dimension,
hence there exists a section $u$ of $\mathcal{L}$ such that $D:=u^{-1}(0)$ intersects $C$ at points $z_j$ with multiplicity at least $p_j$.
Since $D\cdot C=\langle c_1(\mathcal{L}),C\rangle=t$ by assumption, we must have $C\subset D$. But this contradicts the assumption that
no section of $\mathcal{L}$ vanishes entirely on $C$. This implies that $\sum\#S_j\cap[0,p_j)\ge\dim U$. In notation from 
Section~\ref{sec:semigroups} we write this as 
\[\sum R_j(p_j)\ge \dim U.\]  
Taking the infimum over all possible numbers $p_j$ satisfying $\sum p_j=p$ we infer that $R(p)\ge\dim U$.
This is exactly the statement of Theorem~\ref{thm:generalsurface}.
\end{proof}
\begin{proof}[Proof of Lemma~\ref{lem:oncodim}]
Notice that by definition of the semigroup we have $m_j(u)\in S_j$ for any $j$.
It follows that it is enough to show the lemma only for $p_1\in S_1,\ldots,p_n\in S_n$.
We proceed by induction on $\nu:=\sum\#S_j\cap[0,p_j)$. If $\nu=0$, it follows that $p_1=\ldots=p_n=0$, so there is nothing to prove.
Suppose the lemma holds for $p_1,\ldots,p_n$ and let $p_1'$ be the next element in $S_1$ after $p_1$ (the argument works
for arbitrary $z_j$; we fix $j=1$ for simplicity). If $\nu\ge \dim U-1$, 
then the statement of Lemma~\ref{lem:oncodim}
is that
\[\codim V(p_1',p_2,\ldots,p_n)\le \nu+1=\dim U,\]
which holds always if $V$ is not empty. So the only relevant case is $\nu<\dim U-1$. 
The condition $m_1(u)>p_1$ defines a subspace of $V(p_1,\ldots,p_n)$ of codimension at most $1$. But $m_1(u)>p_1$ implies that $m_1(u)\ge p_1'$. This
means that
\[\codim (V(p_1',p_2,\ldots,p_n)\subset V(p_1,\ldots,p_n))\le 1.\]
and the codimension of $V(p_1',p_2,\ldots,p_n)$ in $U$ is at most one greater than the codimension of $V(p_1,\ldots,p_n)$. Since 
$\# S_1\cap[0,p_1')=1+\# S_1\cap[0,p_1)$, the induction step is accomplished.
\end{proof}

\begin{example}\label{ex:notsogeneralsurface}
As an application of Theorem~\ref{thm:generalsurface} let us discuss the case where $X=X_e$ is a Hirzebruch surface and $C$ is of type $(a,b)$. 
Let us choose
$s_1>0$ $s_2\ge 0$ such that $s_1<a$, $s_2<b$. Let $\mathcal{L}=\mathcal{O}(s_1L+s_2M)$. It is clear that no section of $\mathcal{L}$ vanishes on $C$. Otherwise, the sheaf
$\mathcal{O}((s_1-a)L+(s_2-b)M)$ would have a section, but this is impossible, since the assumptions $s_1<a$, $s_2<b$ guarantee that
$\mathcal{O}((a-s_1)L+(b-s_2)M)$ is very ample; see 
\cite[Corollary~V.2.18]{Hart:1977}. We have
\[t=\langle c_1(\mathcal{L}),C\rangle=(s_1L+s_2M)\cdot (aL+bM)=s_1b+s_2(a+be).\]
As $s_1>0$, $s_2\ge 0$, we also have
\[\dim \Gamma(X_e,\mathcal{L})=\chi(\mathcal{L}).\]
The last quantity can be computed using the Riemann--Roch theorem. We obtain
\[\chi(\mathcal{L})=P(s_1,s_2).\]
Theorem~\ref{thm:generalsurface} now gives the same inequality as Theorem~\ref{THF}.
\end{example}

\begin{remark}
In Example~\ref{ex:notsogeneralsurface} the range of $s_1,s_2$ is slightly different than the range in Theorem~\ref{THF}. 
It is not hard to extend the range of $s_1,s_2$ using for instance \cite[Proposition~4.3.3]{COX2}, or formula in \cite[page 18]{MOEPHD} 
to compute $\dim \Gamma(X,\mathcal{L})$ as well as checking that
$\mathcal{L}$ does not admit a section vanishing entirely on $C$. We did not extend it in Example~\ref{ex:notsogeneralsurface} because it
is only an alternative proof of Theorem~\ref{THF}.
\end{remark}

\section{The link at infinity for Hirzebruch surfaces}\label{Spectra}

The main goal of Section~\ref{Spectra} and Section~\ref{S} is to establish Theorem~\ref{thm:spectrum}. We first pass to describing the link of a curve at infinity. 

Let $C$ be a curve of type $(a,b)$ in a Hirzebruch surface $X_e$.
Let $L$ be a vertical line and $M_0$ the special section. Then $X_e\setminus(L\cup M_0)=\mathbb{C}^2$. Let $N$ be a tubular neighbourhood of $L\cup M_0$.
Then $\partial N\cong S^3$ and the intersection of $C$ with $\partial N$ is a candidate for a link of $C$ at infinity. 
In our computations we assume for simplicity that $C$ intersects
both $L$ and $M_0$ transversally. Making the intersection of $C$ transverse to $L$ is simple: we can choose $L$ to be a transverse fibre and the set
of transverse fibres is open--dense. Yet in order to make $C$ transverse
to $M_0$ we might need to perturb $M_0$ in the \emph{smooth category}. This will not affect our reasoning, because the link at infinity will be shown to be
a topological invariant and the proof of the semicontinuity of the spectrum works even in the topological category.

In this approach we will follow the ideas of \cite{Neu-inf}.

\subsection{The link at infinity; a first glance in the plane}\label{sec:linkatinf1}
To illustrate our computations we begin with a rather standard example of computing the link at infinity of a quasi--homogeneous curve in $\C^2$.

So let $C$ be given by $x^p-y^q=0$ in $\C^2$ with $q>p$, $\gcd(p,q)=1$. Its link at infinity is clearly the torus knot $T(p,q)$. We will show an alternative
approach for computing this link, following essentially \cite{Neu-inf}.

First let us study the intersection of $C$ with the line at infinity. We choose coordinates $z=\frac{1}{x}$ and $u=\frac{y}{x}$ and the line at infinity
in these coordinates is given by $z=0$. The equation $x^p-y^q=0$ transforms into 
\[\frac{1}{z^p}-\frac{u^q}{z^q}=0,\]
that is
\[z^{q-p}-u^q=0.\]
Take a small ball  $B$ with a centre at $(z,u)=(0,0)$ and let $S=\partial B$. 
The line at infinity intersects $S$ along an unknot, while $C$ intersects $S$ along the torus knot $T(q,q-p)$,
that is the link of the singularity on the curve $z^{q-p}-u^q=0$; see Figure~\ref{fig:neumann}(left). The unknot and the torus knot are presented 
schematically in Figure~\ref{fig:schematic1}.
This torus knot has linking number $q$ with the unknot, corresponding to the fact that the intersection number of $C$ with $L$ is equal to $q$.

Adding a two--handle to $B$ along the unknot with framing $1$ yields the tubular neighbourhood of the line at infinity; see Figure~\ref{fig:neumann}(right),
the framing is exactly the self--intersection of the line at infinity.
The boundary of this tubular neighbourhood
is a large sphere in $\C^2$, that is, the complement of the line of infinity seen from outside. On the other hand, this sphere is also the result of a $+1$ surgery on the unknot.

\begin{figure}
\begin{pspicture}(-5,-3.8)(5,1)
\rput(-3,0){%
\pscircle[linestyle=dashed,fillstyle=solid,opacity=0.2,fillcolor=gray](0,0){0.6}\rput(0,0.8){\psscalebox{0.7}{$S$}}
\pscircle[linestyle=solid,linewidth=1.5pt](0,-1.8){1.8}\rput(1.5,-1.8){\psscalebox{0.7}{$L_\infty$}}
\psbezier(-1,-1)(-0.2,-1)(0,-0.4)(0,0)
\psbezier(0,0)(0,-0.4)(0.2,-1)(1,-1)\rput(1,-0.8){\psscalebox{0.7}{$C$}}}
\rput(3,0){%
\pscircle[linestyle=dotted,fillstyle=solid,opacity=0.2,fillcolor=gray](0,-1.8){2.1}
\pscircle[linestyle=dotted,fillstyle=solid,opacity=1.0,fillcolor=white](0,-1.8){1.5}
\pscircle[linestyle=solid,linewidth=1.5pt](0,-1.8){1.8}
\psbezier(-1,-1)(-0.2,-1)(0,-0.4)(0,0)
\psbezier(0,0)(0,-0.4)(0.2,-1)(1,-1)\rput(1,-1.2){\psscalebox{0.7}{$C$}}
}
\end{pspicture}
\caption{Left: The line at infinity $L_\infty$ and the link of the singularity of $C$ at infinity. A similar picture appears in \cite[page 462]{Neu-inf}. Right: Adding a two--handle to the neighbourhood of the point at infinity yields a tubular neighbourhood of the line at infinity. Its boundary is a large sphere in $\C^2$.}\label{fig:neumann} 
\end{figure}
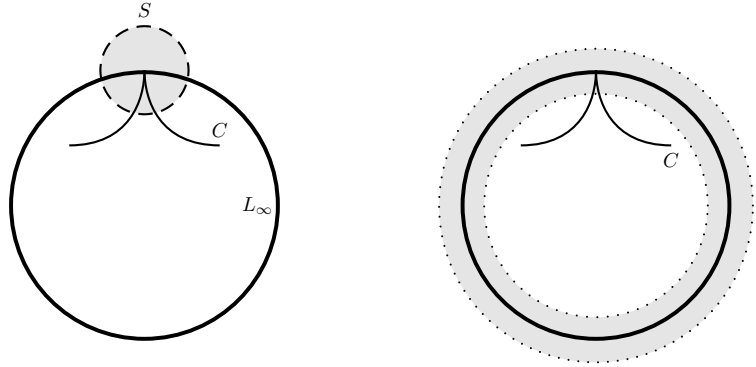

\begin{figure}
\begin{pspicture}(-5,-2)(5,2)
\necklacetemplate{1}
\pscircle[linewidth=1.5pt](0,0){\necklaceradius}\crosslink{$T(q-p,q)$}
\end{pspicture}
\caption{A schematic presentation of the link from 
Figure~\ref{fig:neumann} (left part). The unknot is the intersection of the line at infinity with $S$, while the torus knot, represented
here by a segment, is the intersection of $C$ with $S$.}\label{fig:schematic1}
\end{figure}
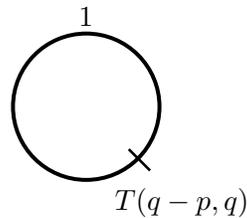

The $+1$ surgery on the unknot, that is, blowing down the unknot with framing $+1$ means that the torus knot $T(q-p,q)$ acquires a negative Dehn
twist along the longitude, so it becomes the torus knot $T(-p,q)$ in the large sphere. This sphere is seen from outside, so we need to reverse the orientation
obtaining the link $T(p,q)$. This is the link of $C$ at infinity.

\subsection{Plumbing at infinity and the Nagata transform}

On a Hirzebruch surface we need to find an analogue to the notion of a line at infinity. We choose one particular fibre $L$ and a horizontal section $M'$,
which is a smooth rational curve isotopic to the special section, for example we can take $M'$ to be a smooth perturbation of the special section $M_0$.
We have $L^2=0$, $L\cdot M'=1$ and $M'^2=-e$. Consider $N$, the tubular neighbourhood of $L\cup M'$ and $Y=\p N$. 
The complement $X_e\setminus N$ is a four--ball, and $\partial N$ is a three--sphere. As a boundary of a neighbourhood of $L\cup M'$,
$N$ is a plumbed 4--manifold
with a plumbing graph as depicted in Figure~\ref{fig:plumbing1}(left).
It will be more
convenient to use a surgery description of a plumbed manifold as on the right hand side of Figure~\ref{fig:plumbing1}.

There exists a birational map between $X_e$ and $X_{e-1}$, which is a sequence of one $(-1)$ blow--up and $(-1)$ blow--down. 
In algebraic geometry, this birational map is called the \emph{Nagata transform} (sometimes referred to as an \emph{elementary transformation}); see Figure~\ref{fig:nagata} for an illustration.
Successive Nagata transforms of the Hirzebruch surface $X_e$ yields, after a finite number of steps, the situation in Figure~\ref{fig:laststage}.
Contracting the $-1$ curve yields an unknot with framing $1$ 
(this corresponds to the well--known fact that $X_1$ is $\C P^2$ blown up in one point). Contracting this 
unknot we obtain $S^3$.

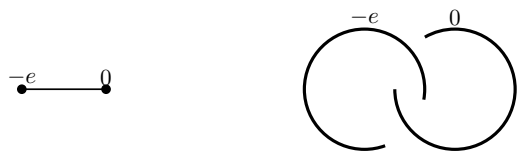
\begin{figure}

\begin{pspicture}(-5,-1)(5,1)
\rput(-2,0){\psscalebox{0.8}{
\psline(-0.7,0)(0.7,0)
\pscircle[fillcolor=black,fillstyle=solid](-0.7,0){0.08}
\pscircle[fillcolor=black,fillstyle=solid](0.7,0){0.08}
\rput(0.7,0.2){$0$}
\rput(-0.7,0.2){$-e$}}}
\rput(2,0){\psscalebox{0.8}{\blackdoublenecklace{-e}{0}}}
\end{pspicture}
\caption{The plumbing diagram of the regular neighbourhood of a 'line at infinity' of a Hirzebruch surface and the corresponding surgery presentation
of the three sphere}\label{fig:plumbing1}
\end{figure}

\begin{figure}
\begin{pspicture}(-7,-1.0)(7,1.72)
\rput(-5.5,0){\psscalebox{0.8}{\contractleftnecklace{-e}{0}}}
\rput(0.2,0){\psscalebox{0.8}{\colortriplenecklace{-e}{-1}{-1}}}
\rput(5,0){\psscalebox{0.8}{\contractmiddlenecklace{-(e-1)}{0}}}
\psline{->}(2.7,0)(4.0,0) 
\psline{->}(-2.0,0)(-3.2,0)
\end{pspicture}
\caption{Nagata transform at the level of links at infinity. The surgery presentation on the left corresponds to $X_e$, it is blown up to the diagram
at the center and then the middle component is contracted, yielding a presentation corresponding to $X_{e-1}$.}\label{fig:nagata}
\end{figure}
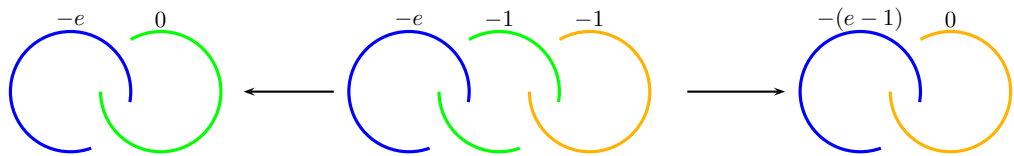

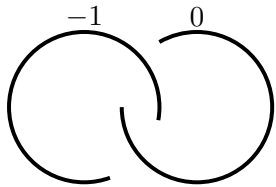
\begin{figure}\begin{pspicture}(-5,-1.3)(5,1.5)
\blackdoublenecklace{-1}{0}
\end{pspicture}
\caption{Last stage of the Nagata transform}\label{fig:laststage}
\end{figure}

\subsection{The link at infinity of curves in Hirzebruch surfaces}\label{s:linkhirz}

Now we want to use the Nagata transform to describe the link at infinity of a given curve in a Hirzebruch surface. For simplicity we restrict to the case when $C$ is a type $(a,b)$ curve
in $X_e$ intersecting $L$ and $M'$ transversally. By definition $C\sim aL+bM=aL+b(M_0+eL)$, hence $C\cdot L=b$ and $C\cdot M'=C\cap M_0=a+be-be=a$.
Let $z_1,\ldots,z_b=C\cap L$ and $w_1,\ldots,w_a=C\cap M'$. Choose a ball $B$ lying in a small tubular neighbourhood of $L\cap M'$ and containing all
the points $z_1,\ldots,z_b$, $w_1,\ldots,w_a$ and $L\cap M'$ as in Figure~\ref{fig:neumannhirzebruch}.
\begin{figure}
\begin{pspicture}(-5,-2)(5,1.5)
\psccurve[fillcolor=gray,opacity=0.1,fillstyle=solid,linestyle=dashed](-1.5,-1.8)(-1.2,-1.5)(-1.1,-0.7)(0.0,-0.3)(2.0,-0.2)(2.3,0.0)(2.0,0.2)(-1.1,0.4)(-1.8,0.35)(-1.9,0.0)(-1.8,-0.7)(-1.8,-1.4)\rput(-0.5,0.6){\psscalebox{0.8}{$S$}}
\psline(-3,0)(3,0)\rput(2.8,0.2){\psscalebox{0.7}{$M'$}}
\psline(-1.5,-2)(-1.5,1)\rput(-1.3,0.8){\psscalebox{0.7}{$L$}}
\psarc(-1,0){2}{340}{20}
\psarc(-1.5,0){2}{340}{20}
\psarc(-0.5,0){2}{340}{20}
\pscircle[fillcolor=black,fillstyle=solid](1,0){0.07}\rput(1.2,0.2){\psscalebox{0.7}{$w_2$}}
\pscircle[fillcolor=black,fillstyle=solid](1.5,0){0.07}\rput(1.7,0.2){\psscalebox{0.7}{$w_3$}}
\pscircle[fillcolor=black,fillstyle=solid](0.5,0){0.07}\rput(0.7,0.2){\psscalebox{0.7}{$w_1$}}
\psarc(-1.5,1.5){3}{255}{285}
\psarc(-1.5,2){3}{255}{285}
\pscircle[fillcolor=black,fillstyle=solid](-1.5,-1.5){0.07}\rput(-1.3,-1.3){\psscalebox{0.7}{$z_2$}}
\pscircle[fillcolor=black,fillstyle=solid](-1.5,-1.0){0.07}\rput(-1.3,-0.8){\psscalebox{0.7}{$z_1$}}
\pscircle[fillcolor=black,fillstyle=solid](-1.5,-0.0){0.07}
\end{pspicture}
\caption{Figure~\ref{fig:neumann} revisited. This time the ball $B$ (the shaded region) captures all the intersection points of $C$ with $L\cup M'$, as well as $L\cap M'$.%
}\label{fig:neumannhirzebruch}
\end{figure}
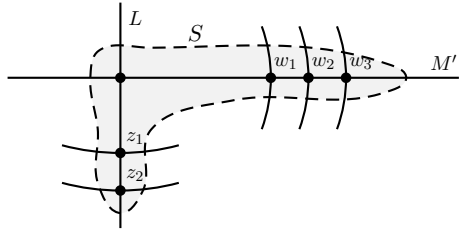
Let $S=\partial B$. The intersection of $L\cup M'$ with $S$ forms a Hopf link. On gluing two--handles to $B$ along this link, with framings respectively $L^2$
and ${M'}^2$ we obtain a tubular neighbourhood of $L\cup M'$. The boundary of this tubular neighbourhood is a large sphere in $\C^2=X_e\setminus(L\cup M')$
seen from outside. This large sphere is thus the effect of a $(0,-e)$ surgery on the Hopf link; compare with Figure~\ref{fig:plumbing1}.

The intersection $C\cap S$ consists of $a+b$ components, corresponding to points $z_1,\ldots,z_b$ and $w_1,\ldots,w_a$. Since $C$ is smooth at all these
points and it is transverse to $L\cup M'$, 
each component of the intersection is an unknot. The unknots corresponding to $z_1,\ldots,z_b$ have linking number $1$ with the link $L\cap S$,
while the unknots corresponding to $w_1,\ldots,w_a$ have linking number $1$ with $M'\cap S$. To be consistent with Section~\ref{sec:linkatinf1},
we remark, that the first set of meridians forms the torus link $T(0,b)$, while the second set forms the torus link $T(0,a)$.
We present these torus links schematically in Figure~\ref{fig:notation}, where we introduced also integer parameters $x$ and $y$. At the beginning we have $x=y=0$,
but later we will have to allow that $x,y\neq 0$.

Our aim is to see these two torus links in a standard sphere, that is to contract the framed Hopf link
in Figure~\ref{fig:plumbing1}(right). This is performed inductively and the induction step is the Nagata transform. We will
explain now, how the two torus links in Figure~\ref{fig:notation} change under the Nagata transform. 

Start with the intermediate stage of the Nagata transform; in Figure~\ref{fig:nagata2} it is the middle link.
Blowing down either the middle circle or the right circle in this link (this corresponds to going to the right link and to the left link, respectively)
does not affect $T(ax,a)$. If the middle circle is blown down, the link $T(by,b)$ also remains unchanged.
However, if we blow down the right circle,
the meridians of the right circle get an additional twist, as well as a clasp with the middle
circle, so that $T(by,b)$ becomes $T(b(y+1),b)$. The Nagata transform corresponds to going from the left link to the right link in Figure~\ref{fig:nagata2},
so that $T(b(y+1),y)$ becomes $T(by,b)$, $T(ax,a)$ is unchanged and the framing of the left circle is changed from $-e$ to $-(e-1)$.

\begin{figure}
\begin{pspicture}(-5,-1.5)(5,1.5)
\blackdoublecnecklace{-e}{0}{$T(ax,a)$}{$T(by,b)$}
\end{pspicture}
\caption{The link consisting of $a$ meridians going around $-e$ framed circle with $x$ twists and $b$ meridians of the $0$ framed circle with $y$ full twists.}\label{fig:notation}
\end{figure}
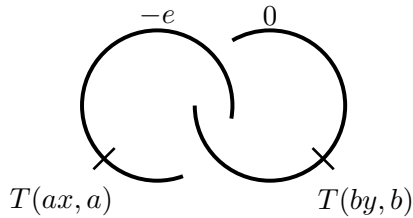

\begin{figure}\begin{pspicture}(-7,-1)(7,1)
\rput(0.2,0){\psscalebox{0.7}{\colortriplecnecklace{-e}{-1}{-1}{$T(ax,a)$}{$T(by,b)$}}}
\rput(-5.5,0){\psscalebox{0.7}{\contractleftcnecklace{-e}{0}{$T(ax,a)$}{$T(b(y+1),b)$}}}
\rput(5,0){\psscalebox{0.7}{\contractmiddlecnecklace{-(e-1)}{0}{$T(ax,a)$}{$T(by,b)$}}}
\psline[linecolor=orange]{->}(2.7,0)(4.0,0)
\psline[linecolor=green]{->}(-2,0)(-3.2,0)
\end{pspicture}
\caption{The Nagata transform and its effect on the link consisting of $a$ meridians of the $-e$ framed circle and $b$ meridians
of the $0$ framed circle.}\label{fig:nagata2}
\end{figure}
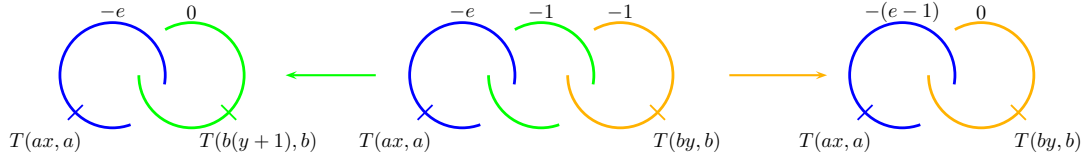

After $e-1$ subsequent Nagata transforms, the framing of the left circle becomes $-1$. We blow down the circle obtaining a link, with only one circle (with
framing $+1$)
and torus links $T(a,a)$ and $T((1-e)b,b)$ linked to it; see Figure~\ref{fig:stage2}.

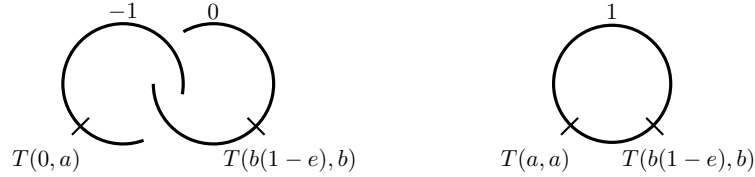
\begin{figure}\begin{pspicture}(-5,-1)(5,1)
\rput(-2.5,0){\psscalebox{0.8}{\blackdoublecnecklace{-1}{0}{$T(0,a)$}{$T(b(1-e),b)$}}}
\rput(4.0,0){%
\psscalebox{0.8}{\necklacetemplate{$1$}
\pscircle(0,0){\necklaceradius}
\crosslink{$T(b(1-e),b)$}
\othercrosslink{$T(a,a)$}}}
\end{pspicture}
\caption{On the left: the link of $C$ at infinity after $e-1$ Nagata transforms. On the right, this link after blowing down the $-1$ curve.}\label{fig:stage2}
\end{figure}

Blowing down the circle with framing $+1$ on the right of Figure~\ref{fig:stage2} means that the two torus links will acquire a negative twist. That
is, the $T(a,a)$ torus link will become the $T(0,a)$ torus link again, while $T(b(1-e),b)$ will become the $T(-be,b)$ torus link. The two torus links are now
clasped: each of the components of $T(0,a)$ with each of the components of $T(-be,b)$ will form a negative Hopf link.
As in the case of the link at infinity of a curve in $\C P^2$, we have to reverse the orientation. In this way the $T(-be,b)$ will become $T(be,e)$, while $T(0,a)$ will remain
$T(0,a)$. The negative Hopf link will become the positive Hopf link. The final result is shown in Figure~\ref{fig:link}.

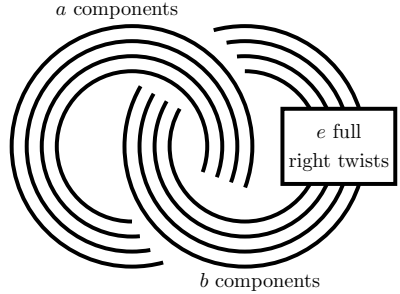
\begin{figure}
\begin{pspicture}(-5,-2)(5,2)
\psset{linewidth=1.5pt}
\psarc(-1,0){1}{340}{270}
\psarc(-1,0){1.2}{340}{275}
\psarc(-1,0){1.4}{340}{280}
\psarc(-1,0){1.6}{340}{285}
\rput(-1.2,1.8){\psscalebox{0.7}{$a$ components}}
\psarc(0.5,0){1}{150}{90}
\psarc(0.5,0){1.2}{150}{95}
\psarc(0.5,0){1.4}{150}{100}
\psarc(0.5,0){1.6}{150}{105}
\pspolygon[fillstyle=solid,fillcolor=white,linecolor=black](1,-0.5)(2.5,-0.5)(2.5,0.5)(1,0.5)
\rput(1.75,0.2){\psscalebox{0.7}{$e$ full}}
\rput(1.75,-0.2){\psscalebox{0.7}{right twists}}
\rput(0.7,-1.8){\psscalebox{0.7}{$b$ components}}
\end{pspicture}
\caption{The link at infinity of a curve of degree $(a,b)$.}\label{fig:link}
\end{figure}

\begin{definition}
The link in Figure~\ref{fig:link} is called the \emph{link at infinity} of a curve of degree $(a,b)$. We denote this link by $L_{a,b}$.
\end{definition}

\section{Hermitian Variation Structure of the link at infinity}\label{S}

In this section we determine the spectrum of the curve $C$ at infinity using the language of Hermitian Variation Structures for links as introduced in \cite{BN-hodge}.
Hermitian Variation Structure is encoded in so called Hodge numbers $p^k_\lambda(u)$, where $k=1,2,\ldots,$, with $\lambda\in S^1$ and $u\in \{\pm 1\}$. Here $\lambda$ corresponds
to roots of the Alexander polynomial of the link; and since 
the Alexander polynomial has roots only at the unit circle, the Hodge numbers $q^k_\lambda$ for $|\lambda|<1$
are not present. In the following we will gather enough data about $p^k_\lambda(u)$ to be able to compute the spectrum of the link at infinity. We will do this in four steps. First we will compute the Alexander polynomial of $L_{a,b}$ and show that the $p^k_\lambda(u)$ vanish for $k>2$ and for $k=2$ if $\lambda=1$.
Second, we compute the equivariant signatures of the link at infinity. Third, we use the equivariant signatures and the Alexander polynomial to gather enough conditions on the Hodge numbers  
to recover the non--integer part of the spectrum. In the last step the Hodge numbers for
$\lambda=1$ will be determined from the linking matrix of the components of $L_{a,b}$ as in \cite[Section 3]{BN-spec}. 

The proof of Theorem~\ref{thm:spectrum} is presented at the end of this section.

\subsection{Splice presentation of $L_{a,b}$ and its first consequences}
After a series of observations about the link $L_{a,b}$, our first result on the Hodge numbers is given in Lemma~\ref{lem:monod}.

We begin with the following observation.
\begin{lemma}
The link at infinity $L_{a,b}$ is a graph link. Its Eisenbud--Neumann diagram is as on Figure~\ref{fig:linka1}.
\end{lemma}

We denote by $L^1$ and $L^2$ the splice components of $L_{a,b}$. On Figure~\ref{fig:linka1} the component $L^1$ is presented on the left, while the
component $L^2$ is on the right. For these components we have the following estimates.

\begin{lemma}
The multiplicity of the multilink $L^1$ is equal to $b$, while the multiplicity of $L^2$ is equal to $w=a+be$.
\end{lemma}
\begin{proof}
The results follow from straightforward computations; see for example \cite{Neu-other}.
\end{proof}

Then we make the following observation about $L_{a,b}$.
\begin{corollary}
The link $L_{a,b}$ is a fibred link.
\end{corollary}
\begin{proof}
Because the multiplicity of every node is nonnegative, this follows from \cite[Theorem 11.2]{EN}.
\end{proof}

Moreover, \cite[Theorem 11.3]{EN} allows us to explicitly compute the Alexander polynomial of $L_{a,b}$.
\begin{lemma}\label{lem:Alexander}
The Alexander polynomial of $L_{a,b}$ is equal to
\[\Delta(t)=(t-1)(t^w-1)^{b-1}(t^b-1)^{a-1}.\]
\end{lemma}

The fact that $L_{a,b}$ is a graph link affects the Hodge numbers in the following way.

\begin{lemma}[The monodromy theorem; see \expandafter{\cite[Section 13]{EN}}{} or \cite{Neu-other}]\label{lem:monod}
We have $p^k_\lambda(u)=0$ for $k>2$ and for $k=2$ and $\lambda=1$. Moreover the only positive values of $p^k_\lambda(u)$ can appear if $\lambda$ is
a root of unity of order $w$ or $a$.
\end{lemma}

The Alexander polynomial is the characteristic polynomial of the monodromy $h$ acting on $H_1(F)$, where $F$
is the fibre of the fibration of the complement to $L_{a,b}$. By Lemma~\ref{lem:monod} the monodromy has Jordan block of size at most $2$. We can compute the Jordan
blocks of size $2$ from the characteristic polynomial of the monodromy operator restricted to the image $h^c-1$, where $c=\gcd(w,b)$.
The algorithm in \cite[Section 14]{EN} or \cite{Neu-other}  allows us to compute this as well. We obtain
\[\Delta_2(t)=\frac{t^c-1}{t-1}.\]
This polynomial affects the Hermitian Variation Structure related to the link $L_{a,b}$, but it does not affect its spectrum.

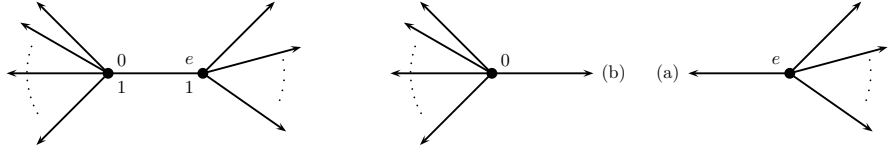
\begin{figure}
\begin{pspicture}(-5,-1)(5,2)
\rput(-3,0){\psscalebox{0.9}{%
\nodedot{(-0.7,0)}\nodedot{(0.7,0)}
\psline(-0.7,0)(0.7,0)
\psline{<-}(-2.2,0)(-0.7,0)
\rput(0.5,0.2){\psscalebox{0.7}{$e$}}
\rput(0.7,0){%
\rput{45}(0,0){\psline{->}(0,0)(1.5,0)}
\rput{15}(0,0){\psline{->}(0,0)(1.5,0)}
\rput{325}(0,0){\psline{->}(0,0)(1.5,0)}
\psarc[linestyle=dotted,fillstyle=none](0,0){1.2}{340}{10}}
\rput(-0.5,0.2){\psscalebox{0.7}{$0$}}
\rput(-0.5,-0.2){\psscalebox{0.7}{$1$}}
\rput(0.5,-0.2){\psscalebox{0.7}{$1$}}
\rput(-0.7,0){%
\rput{135}(0,0){\psline{->}(0,0)(1.5,0)}
\rput{150}(0,0){\psline{->}(0,0)(1.5,0)}
\rput{225}(0,0){\psline{->}(0,0)(1.5,0)}
\psarc[linestyle=dotted,fillstyle=none](0,0){1.2}{160}{210}}
}}
\rput(3,0){\psscalebox{0.9}{%
\nodedot{(-1.7,0)}\nodedot{(2.7,0)}
\psline{->}(-1.7,0)(-0.2,0)\rput(0.1,0){\psscalebox{0.7}{(b)}}
\psline{->}(2.7,0)(1.2,0)\rput(0.9,0){\psscalebox{0.7}{(a)}}
\psline{<-}(-3.2,0)(-1.7,0)
\rput(2.5,0.2){\psscalebox{0.7}{$e$}}
\rput(2.7,0){%
\rput{45}(0,0){\psline{->}(0,0)(1.5,0)}
\rput{15}(0,0){\psline{->}(0,0)(1.5,0)}
\rput{325}(0,0){\psline{->}(0,0)(1.5,0)}
\psarc[linestyle=dotted,fillstyle=none](0,0){1.2}{340}{10}}
\rput(-1.5,0.2){\psscalebox{0.7}{$0$}}
\rput(-1.7,0){%
\rput{135}(0,0){\psline{->}(0,0)(1.5,0)}
\rput{150}(0,0){\psline{->}(0,0)(1.5,0)}
\rput{225}(0,0){\psline{->}(0,0)(1.5,0)}
\psarc[linestyle=dotted,fillstyle=none](0,0){1.2}{160}{210}}
}}%
\end{pspicture}
\caption{The link $L_{a,b}$ and its splice components. There are $a$ arrowheads on the left and $b$ arrowheads on the right.}\label{fig:linka1}
\end{figure}

\subsection{The equivariant signatures of the link at infinity}
In this part we use Neumann's algorithm, see \cite{Neu-splice}, to compute the equivariant signatures of $L^1$ and $L^2$ for $\lambda\in S^1\setminus\{1\}$. The
equivariant signatures are additive, hence the signature of $L_{a,b}$ is the sum of the signatures for $L^1$ and $L^2$.

For $x\in(0,1)$ we denote by $\sigma^1_x$, respectively $\sigma^2_x$, the equivariant signature of $L^1$, respectively $L^2$, corresponding
to the value $\lambda=e^{2\pi ix}$.

The following lemma is the core result.
\begin{lemma}\label{lem:sign}
For $p=1,\ldots,w-1$, the signature is equal to 
\[\sigma_{p/w}^1=2\left\lfloor\frac{pb}{w}\right\rfloor-(b-1)-\delta,\]
where $\delta=1$ if $w$ divides $pb$, otherwise it is $0$. Similarly, for $q=1,\ldots,b-1$, the signature
\[\sigma_{q/b}^2=2\left\lfloor\frac{qa}{b}\right\rfloor-(a-1)-\delta',\]
where $\delta'=1$ if $b$ divides $qa$.
\end{lemma}
\begin{proof}
We use the notation from \cite{Neu-splice}. The two parts (for $L^1$ and $L^2$) are analogous, we will prove only the first one. The second part
can be deduced from the first by swapping the roles of $a$ and $b$ and setting $e=0$.

The (multi)link $L^1$ has $b$ arrowheads with multiplicity $1$ and one arrowhead with multiplicity $a$. 
Let us set $m_1=\ldots=m_b=1$, $m_{b+1}=a$. Furthermore $\alpha_1=\ldots=\alpha_{b}=1$ and $\alpha_{b+1}=e$, 
The numbers $\beta_1,\ldots,\beta_{b+1}$
are defined by the condition 
\[\beta_j\alpha_1\ldots\widehat{\alpha_j}\ldots \alpha_{b+1}\equiv1\bmod \alpha_{j},\] 
hence 
$\beta_1=\ldots=\beta_b=0$ and $\beta_{b+1}=1$. Therefore $w=a+be$
is the multiplicity, and we have \[s_1=\ldots=s_{b}=\frac{(m_1-\beta_1w)}{\alpha_1}=1 \textrm{ and } s_{b+1}=\frac{(a-(a+ke))}{e}=-b.\]

Suppose now that $p=1,\ldots,w-1$. By \cite[Theorem 5.3]{Neu-splice} the equivariant
signature of the splice component at $e^{2\pi ip/w}$ is equal to
\[
\sigma^1_{p/w}=2\sum_{j=1}^{b+1}\sawtooth{\frac{s_jp}{w}}=2b\sawtooth{\frac{p}{w}}-2\sawtooth{\frac{pb}{w}},
\]
where $x\mapsto\langle x\rangle$ is the sawtooth function, that is
\begin{equation}\label{eq:sawtoothfunction}
\langle x\rangle =\begin{cases} \{x\}-\frac12 & x\not\in\Z \\ 0 & x\in\Z.\end{cases}
\end{equation}

The expression for $\sigma^1_{p/m}$ can be rewritten as
\begin{equation}\label{eq:sigform}
\sigma^1_{p/w}=2\left\lfloor\frac{pb}{w}\right\rfloor-(b-1)-\delta,
\end{equation}
where $\delta=1$ if $w|pb$ (that is, if $b\frac{p}{w}$ is an integer), otherwise $\delta=0$. This proves the first part.
\end{proof}

We then observe that since $b$ divides $qa$ if and only if it $b|qw$, the case $\delta=1$ is equivalent to $\delta'=1$. Hence this happens only if $x$ can be written as $p/w$ and $x$ can be written as $q/b$ for some integers $p$ and $q$. Note that this observation is reflected in the structure of Table~\ref{table:one}: the
contribution of an element $x$ that can be written both as $\frac{p}{w}$ and as $\frac{q}{b}$ (second row in Table~\ref{table:one}) is not
merely a sum of the contribution from the fourth and the sixth row.

\subsection{The part $\lambda\neq 1$ of the spectrum}
Now to the third step in our process, where we relate the Hodge numbers to the Alexander polynomial and to the equivariant signatures for $\lambda\neq 1$.
\begin{remark}
For the reader's convenience, in Section~\ref{signissue} we discuss the sign conventions used throughout this section. We also present a sample
computation of the spectrum.
\end{remark}
We begin with the following observation; see \cite[Section 4.1]{BN-hodge}.

\begin{equation}\label{eq:sumofhodge}
\sum_{u=\pm 1}\left(p^1_\lambda(u)+2p^2_\lambda(u)\right)=\ord_{t=\lambda}\Delta(t)=\begin{cases} 
b-1&\textrm{if $\lambda=e^{2\pi i p/w}$}\\ 
a-1 &\textrm{if $\lambda=e^{2\pi i q/b}$}\\
a+b-2 &\textrm{if $\lambda=e^{2\pi i q/b}=e^{2\pi ip/w},$}
\end{cases}
\end{equation}
where we should interpret the expressions like $\lambda=e^{2\pi iq/b}$ as ``there exists $q\in\Z$ such that $\lambda=e^{2\pi iq/b}$''.

On the other hand, the equivariant signature can be computed from Hodge numbers as
\begin{equation}\label{eq:hodgediff}p^1_\lambda(-1)-p^1_\lambda(+1)=\sigma_x^1+\sigma_x^2,
\end{equation}
where $x$ is such that $e^{2\pi i x}=\lambda$.
To complete the picture we note that
\[p^2_{\lambda}(+1)+p^2_{\lambda}(-1)=\ord_{t=\lambda}\Delta_2(t),\]
but we will not need this formula.

Even though the three above formulae are insufficient to determine the full Hermitian Variation Structure of $L_{a,b}$ corresponding to the eigenvalue $\lambda$:
the terms $p^2_\lambda(+1)$ and $p^2_\lambda(-1)$ give the same contribution to all the three formulae. Luckily, their contribution to
the spectrum is also the same.
The first two formulae are enough to recover the spectrum. Let us cite a result from \cite[Section 2.3]{BN-hodge}. 

\begin{proposition}
Let $x\in(0,1)$ and $\lambda=e^{2\pi ix}$. Then the multiplicity of $x$ in the spectrum is equal to
\[A_x:=p^1_\lambda(-1)+p^2_\lambda(+1)+p^2_\lambda(-1)\]
and the multiplicity of $1+x$ in the spectrum is equal to
\[B_x:=p^1_\lambda(+1)+p^2_\lambda(+1)+p^2_\lambda(-1).\]
\end{proposition}

We now observe that $A_x+B_x$ is the order at $\lambda$ of the Alexander polynomial, while $A_x-B_x$ is equal to the equivariant signature.
Hence, knowing the Alexander polynomial from Lemma~\ref{lem:Alexander} and the equivariant signature from Lemma~\ref{lem:sign} we can compute explicitly $A_x$ and $B_x$, that
is, find the spectrum. The results of the computations is presented in Table~\ref{table:one}.
We omit the details here.

\subsection{The part $\lambda = 1$ of the spectrum.}
It remains to discuss the case of $\lambda=1$. By Lemma~\ref{lem:monod} we have that $p^2_1(\pm 1)=0$. We also have that 
\[p^1_1(-1)+p^1_1(+1)=\ord_{t=1}\Delta(t)=a+b-1.\]

\begin{lemma} We have $p^1_1(-1)=0$.
\end{lemma}
\begin{proof}
The argument is the same as in \cite[Proposition 3.4.2]{BN-spec}; we will show that the linking form on the subspace of $H_1(F)$ (recall that
$F$ is the fibre of the fibration of the complement to $L_{a,b}$, in particular it is a Seifert surface) spanned by the components of $L_{a,b}$ is negative definite.
Let us denote by $L^1_1,\ldots,L^1_a$, respectively $L^2_1,\ldots,L^2_b$, the components of $L_{a,b}$ lying on a splice component $L^1$, respectively $L^2$. 
We regard them as elements in $H_1(F)$. In $H_1(F)$ these cycles are subject to one relation, namely $L^1_1+\ldots+L^2_b=0$.
Let us consider an element $L$ in the space spanned by $L^1_1,\ldots,L^2_b$, that is
\[L=\sum \gamma_i L^1_i+\sum \delta_jL^2_j\]
And by the arguments of \cite[Section 3]{Neu-other} or \cite[Section 3.7]{BN-spec} the self-linking of $L$ is equal to
\[-\sum_{i<i'} (\gamma_i-\gamma_{i'})^2\lk(L^1_i,L^1_{i'})-\sum_{j<j'}(\delta_j-\delta_{j'})^2\lk(L^2_j,L^2_{j'})-\sum_{i,j}(\gamma_i-\delta_j)^2\lk(L^1_i,L^2_j).\]
This expression is clearly non-positive. If it is zero, then for any $i,j$ we have $\gamma_i=\delta_j$, because $\lk(L^1_i,L^2_j)=1$ for any $i,j$. This implies that
$\gamma_1=\ldots=\gamma_a=\delta_1=\ldots=\delta_b$, that is, the link represents $0\in H_1(F)$. It follows that the linking form is negative definite, so 
the argument in \cite[proof of Theorem 3.4]{Neu-other} implies that $p^1_1(-1)=0$ holds.
\end{proof}

This shows the following result.
\begin{proposition}
The number $1$ enters the spectrum with multiplicity $a+b-1$, while the number $2$ does not belong to the spectrum.
\end{proposition}

In this way the computation of the spectrum of the link at infinity is completed.
\subsection{Some examples}\label{signissue}

It might be quite difficult not to get lost in various conventions. To show that the sign conventions are consistent, 
we look at the positive trefoil, whose signature function
is $-2$ for $z\in(e^{2\pi i/6},e^{2\pi 5i/6})$ and zero outside of the closure of this interval. The input for Neumann's algorithm for the equivariant signature
is $m_1=0,m_2=0,m_3=1$, $\alpha_1=2$, $\alpha_2=3$ and $\alpha_3=1$, so that $w=6$ and $s_1=-3$, $s_2=-4$, $s_3=1$. 
The equivariant signature $2\sum \langle s_jp/q\rangle$ is equal to $-1$ for $p/q=1/6$ and $+1$
for $p/q=5/6$. The spectrum should be $\{\frac56,\frac76\}$ (this is the spectrum of the singularity $x^2+y^3=0$ and we can compute it from the Thom--Sebastiani
formula) and the Hodge numbers are $p^1_{5/6}(-1)=1$ and 
$p^1_{1/6}(+1)=1$; see \cite[Section 5.1]{BN-hodge}. In particular the conventions we use are the following
\begin{itemize}
\item The equivariant signatures are taken as $\sum\langle s_jp/q\rangle$, where $\langle x\rangle$ is the sawtooth function; see \eqref{eq:sawtoothfunction}.
\item The equivariant signature is half the jump of the Tristram--Levine signature. That is half the difference between the right limit and the left limit
of the function $x\to \sigma(e^{2\pi ix})$.
\item The Hodge numbers $p^1_\lambda(+1)$ correspond to negative equivariant signature and the Hodge numbers $p^1_\lambda(-1)$ correspond to positive
values of the equivariant signature. Therefore the equivariant signature is $p^1_\lambda(-1)-p^1_\lambda(+1)$.
\item The Hodge numbers $p^1_\lambda(-1)$ correspond to values in the spectrum in the interval $(0,1)$, whereas the Hodge numbers $p^1_\lambda(+1)$
correspond to values in $(1,2)$.
\end{itemize}

Now let us present a computation of the spectrum.

\begin{example}
Suppose that $a=6$, $b=4$ and $e=0$, moreover $w=6$. The equivariant signatures for $\sigma^1_{i/6}$ are equal to
\[-3,-1,0,1,3\]
The equivariant signatures $\sigma^2_{j/4}$ are equal to
\[-3,0,3\]
By additivity, the equivariant signatures of $L_{4,6}$ at $e^{2\pi ix}$ where 
\[x\in\left\{\frac{1}{6},\frac{1}{4},\frac{1}{3},\frac{1}{2},\frac{2}{3},\frac{3}{4},\frac{5}{6}\right\}\] 
are respectively \[\{-3,-3,-1,0,1,3,3\};\] 
the order of the root of the Alexander polynomials at each of this point is respectively \[\{3,5,3,8,3,5,3\}.\] Hence, the part of the spectrum in $(0,1)$ is
\[\left\{\frac14,\frac13,\frac12,\frac12,\frac12,\frac12,\frac23,\frac23,\frac34,\frac34,\frac34,\frac34,\frac56,\frac56,\frac56\right\},\]
and the part in $(1,2)$ is symmetric. The value $1$ appears in the spectrum with multiplicity \[a+b-1=9.\] The total number of elements in the spectrum
is \[w(b-1)+b(a-1)+1=18+20+1=39.\] Notice that twice the genus of a curve of type $(4,6)$ is equal to $(4-1)(2\cdot 6-2)=30$ and the difference
$39-30=9=a+b-1$.
\end{example}

\subsection{The semicontinuity of the spectrum}
We are now ready to prove Theorem~\ref{thm:spectrum}, comparing the spectrum of the link at infinity to the spectrum of the singular points of $C$. 
\begin{proof}[Proof of Theorem~\ref{thm:spectrum}]
We follow the argument of \cite{BN-spec}.
Let us pick $L$ and $M'$ as in Section~\ref{s:linkhirz}. They intersect $C$ transversally and let $N$ be a tubular neighbourhood of $L\cup M'$. The complement
$B=X_e\setminus  N$ is a standard 4--ball. By definition $L_{a,b}=C\cap \partial B$ is the link of $C$ at infinity. Let $C'=C\cap B$. Let $C''$ be a smoothing
of $C'$.

Suppose $x\in[0,1]$ is such that $\xi:=e^{2\pi ix}$ is not a root of the Alexander polynomials of $L_{a,b}$. Then, by 
\cite[Proposition 2.5.5]{BN-spec} we have
\begin{equation}\label{eq:from_morse}
\begin{split}
-\sigma_{L_{a,b}}(\xi)+(1-\chi(C''))&\ge\sum_{j=1}^k(-\sigma_{K_j}(\xi)+\mu_j)\\
\sigma_{L_{a,b}}(\xi)+(1-\chi(C''))&\ge\sum_{j=1}^k(\sigma_{K_j}(\xi)+\mu_j),
\end{split}
\end{equation}
where $\sigma_L(\xi)$ denotes the Tristram--Levine signature, and $\mu_j$ is the Milnor number.

Now we have that $1-\chi(C'')=b_1(C'')$, but $C''$ is the fibre of the link $L_{a,b}$. So $b_1(C'')$ is just the degree of the Alexander polynomial
of $L_{a,b}$. On the other hand, $\mu_j$, the Milnor number of the singularity of $C$ corresponding to $K_j$, is the degree of the Alexander
polynomial of the link $K_j$.
Therefore we can rewrite \eqref{eq:from_morse} as
\begin{equation}\label{eq:from_morse2}
\mp\sigma_{L_{a,b}}+\deg\Delta_{L_{a,b}}\ge \sum_{j=1}^k\left(\mp\sigma_{K_j}(\xi)+\deg\Delta_{K_j}\right).
\end{equation}

By \cite[Corollary 2.4.6]{BN-spec}, \eqref{eq:from_morse2} is precisely the statement of Theorem~\ref{thm:spectrum}.

\end{proof}
\section{Examples and applications}\label{Examples}

In this section we will show applications of Theorem~\ref{THF} and Theorem~\ref{thm:spectrum} and get results about rational cuspidal curves in Hirzebruch surfaces. In particular, we show that the theorems imply that not all possible (after fundamental results, i.e., the genus formula etc.) cusps can exist on such curves. Moreover, we give explicit constructions of some of the possible curves that pass the obstructions in the theorems.

\subsection{A simple multiplicity estimate}

The following result bounding the multiplicity of a singularity on a curve in a Hirzebruch surface is a trivial consequence of B\'ezout's theorem. We show that for rational cuspidal curves it can also be proven using topological methods.
\begin{proposition}\label{prop:multbound}
Let $r$ be the multiplicity of a singular point on a curve (in general not necessarily rational or cuspidal) 
$C$ of type $(a,b)$ in the Hirzebruch surface $X_e$. Then $r\le b$.
\end{proposition}
\begin{proof}
The standard proof is given for instance in \cite[Theorem 3.1.5]{MOEPHD}.

Suppose now that the curve $C$ is rational and cuspidal.
We set $s_1=1$ and $s_2=0$ in Theorem~\ref{THF} and obtain $R(b+1)\ge 2$, for the $R$ function associated with $C$ (see Definition~\ref{def:Rfunction}). 
Suppose that $z_1,\ldots,z_n$ are singular points of $C$ and that $R_1,\ldots,R_n$
are the corresponding semigroup densities as in \eqref{eq:defR}. Let us consider the point $z_k$, for $k=1,\ldots,n$. Set 
$m_1=\ldots=m_{k-1}=0$, $m_{k+1}=\ldots=m_n=0$ and $m_k=b+1$, so that $\sum m_j=b+1$. By definition of the infimum convolution we get
\[R_1(m_1)+\ldots+R_n(m_n)\ge R(m_1+\ldots+m_n).\]
But $R_1(m_1)=\ldots=R_{k-1}(m_{k-1})=R_{k+1}(m_{k+1})=\ldots=R_n(m_n)=0$. Eventually we obtain
\[R_k(b+1)\ge 2.\]
But if $z_k$ has multiplicity greater than $b$, then
zero is the only element in the semigroup of $z_k$ that is smaller than $b+1$, which leads to the desired contradiction.
\end{proof}

\subsection{Singular points with multiplicity $3$ on $(4,4)$ curves.}

Consider curves of type $(4,4)$ in the Hirzebruch surface $X_e$. Such a curve has genus $6e+9$. We ask, whether such a curve can have a
singularity of type $(3,6e+10)$. Note that the genus formula implies that if a $(4,4)$ curve has such a singularity, then it is rational and cuspidal.

\begin{proposition}
If $e$ is even, then a $(4,4)$ curve in $X_e$ cannot have a $(3,6e+10)$ singularity.
\end{proposition}
\begin{proof}
Set $s_1=1+e/2$ and $s_2=1$. We obtain $s_1b+s_2(a+be)+1=6e+9$. Theorem~\ref{THF} implies that $R(6e+9)\geq 2e+4$. But $R(6e+9)$ is the number
of the elements in the interval $[0,6e+8]$ that belong to the semigroup generated by $3$ and $6e+10$. But this is exactly the set of
all integers in $[0,6e+8]$ divisible by $3$. Its cardinality is $2e+3$, so we get a contradiction.
\end{proof}

The result does not say anything about the case when $e$ is odd. In fact, Theorem~\ref{THF} will not obstruct the existence of a $(3,6e+10)$ singularity
on a $(4,4)$ curve in $X_e$ for $e$ odd.

\begin{remark}
As $4$ is even and $4-1$ is odd, Theorem~\ref{thm:finiteness} below will not obstruct the existence of a $(3,6e+10)$ on a $(4,4)$ curve, regardless of the
parity of $e$.
\end{remark}

\subsection{Curves of type $(6,6)$ with one cusp with one single Puiseux pair in $X_0$}
We now compute the spectra for the link at infinity and the link at the cusp for rational unicuspidal curves of type $(6,6)$ with a single Puiseux pair in $X_0$. We do this for all curves with Puiseux pair that fits the genus formula, and there are three such curves. The computations show that Theorem~\ref{thm:spectrum} obstructs the existence of one of these curves. Moreover, we provide a sketch of the construction of one of the other curves. Note that we cannot say anything about the existence of the remaining curve.

Let $C$ be a unicuspidal curve of type $(6,6)$ in $X_0$, and let the cusp have a single Puiseux pair. We have $a=b=6$, $g=(6-1)(6-1)=25$ and $2g=50$. 
The list of theoretically possible Puiseux pairs is $(2,51)$, $(3,26)$ and $(6,11)$. Using Table~\ref{table:one}, we find: 
\[Sp^\infty_{6,6} = \Bigl{\{}\frac{1}{6}^1,\frac{1}{3}^3,\frac{1}{2}^5, \frac{2}{3}^7, \frac{5}{6}^9, 1^{11}, \frac{7}{6}^9, \frac{4}{3}^7, \frac{3}{2}^5,\frac{5}{3}^3, \frac{11}{6}^1 \Bigr{\}},\] where the exponent denotes the multiplicity of the element.
\begin{enumerate}
 \item[$\mathbf{(2,51)}$] For this Puiseux pair we choose $x=\frac{1}{2}+\epsilon$ for some $\epsilon$ such that $\frac{1}{51}>\epsilon >0$. Then \[\#Sp^\infty_{6,6} \cap (x,x+1)=48.\] On the other hand, \[Sp_{2,51}\cap(x,x+1)=\Bigl{\{}\frac{1}{2}+\frac{1}{51}, \ldots, \frac{1}{2}+\frac{50}{51} \Bigr{\}},\] hence \[\#Sp_{2,51}\cap (x,x+1)=50.\] By Theorem~\ref{thm:spectrum} this is not possible, so such a cusp cannot exist on a curve of type $(6,6)$.
\item[$\mathbf{(3,26)}$] For this Puiseux pair we choose $x=\frac{1}{3}+\epsilon$ for some $\epsilon$ such that $\frac{1}{26}>\epsilon >0$. Then \[\#Sp^\infty_{6,6} \cap (x,x+1)=48.\] On the other hand, \[Sp_{3,26}\cap(x,x+1)=\Bigl{\{}\frac{1}{3}+\frac{1}{26}, \ldots, \frac{1}{3}+\frac{25}{26}, \frac{2}{3}+\frac{1}{26}, \ldots, \frac{2}{3}+\frac{17}{26} \Bigr{\}},\] hence \[\#Sp_{3,26}\cap (x,x+1)=42.\] This does not violate Theorem~\ref{thm:spectrum}.
\item[$\mathbf{(6,11)}$] For this Puiseux pair we choose $x=\frac{1}{6}+\epsilon$ for some $\epsilon$ such that $\frac{1}{11}>\epsilon >0$. Then \[\#Sp^\infty_{6,6} \cap (x,x+1)=44.\] On the other hand, 
\begin{align*}
Sp_{6,11}\cap(x,x+1)=\Bigl{\{}&\frac{1}{6}+\frac{1}{11}, \ldots, \frac{1}{6}+\frac{10}{11}, 
\\&\frac{1}{3}+\frac{1}{11}, \ldots, \frac{1}{3}+\frac{9}{11},
\\&\frac{1}{2}+\frac{1}{11}, \ldots, \frac{1}{2}+\frac{7}{11},
\\&\frac{2}{3}+\frac{1}{11}, \ldots, \frac{2}{3}+\frac{5}{11},
\\&\frac{5}{6}+\frac{1}{11}, \ldots, \frac{5}{6}+\frac{3}{11}   \Bigr{\}},
\end{align*} 
hence \[\#Sp_{6,11}\cap (x,x+1)=34.\] This curve passes the criterion from Theorem~\ref{thm:spectrum}, and it can in fact be constructed by a simple transformation of the plane unicuspidal curve $y^5z-x^6=0$. Indeed, blow up two points on the line $y=0$, which is tangent to the curve at the cusp $(0:0:1)$, and contract its strict transform;
see Example~\ref{ex:brseries} and \cite{MOEPHD}.
\end{enumerate}

\subsection{A finiteness theorem}

As an application of Theorem~\ref{thm:spectrum} we shall prove the following result.

\begin{theorem}\label{thm:finiteness}
Suppose $a,b>0$. Then there is only a finite number of triples $(r,s,e)$ such that
\begin{itemize}
\item $r,s\ge 2$ are coprime, $e\ge 0$;
\item A singularity of type $(r,s)$ occurs on a rational cuspidal curve of type $(a,b)$ in $X_e$;
\item $r<b$ and $r\neq b-1$ if $b$ is even.
\end{itemize}

Put differently, for given $(a,b)$ and sufficiently large $e$ the only possible rational cuspidal curves of type $(a,b)$ in $X_e$ with one singular point having one 
Puiseux pair $(r,s)$ are those with $b=r$ or $b=r+1$ and $b$ even.
\end{theorem}

Before we give the proof of Theorem~\ref{thm:finiteness} we show a constructions of two families of curves in Hirzebruch surfaces with $b=r$. This shows
that Theorem~\ref{thm:finiteness} is close to optimal.

\begin{ex}\label{ex:brseries}
It is possible to construct two series of curves from Theorem~\ref{thm:finiteness} using Cremona transformations of the plane curve $C_d$ 
given by the defining polynomial $x^{d-1}z-y^d$. Note that the curves in the below series exist for all $d \geq3$, and $e,k \geq0$, except $(e,k)=(0,0)$.
\begin{enumerate}
\item Curves of type $(kd,d)$ in $X_e$ with one cusp and Puiseux pair $\bigl(d,(e+2k)d-1\bigr)$.
\item Curves of type $(k(d-1)+1,d-1)$ in $X_e$ with one cusp and Puiseux pair $\bigl(d-1,(e+2k)(d-1)+1\bigr)$.
\end{enumerate}
\end{ex}
\begin{proof}[Sketch of construction]
The constructions of the two series of curves in Example~\ref{ex:brseries} are very similar, hence we sketch only the construction of the first curves and give the initial blowing up for the latter curves. For similar constructions with details, see \cite{MOEPHD}. 

Given the plane unicuspidal curve $C_d$ as above, we blow up a point on the tangent line to the curve at the cusp. This gives in $X_1$ a rational unicuspidal curve of type $(0,d)$ with a cusp with multiplicity sequence $[d-1]$. The fibre through the cusp intersects the curve only at the cusp, with intersection multiplicity $d$. Now, performing subsequent elementary transformations with center on the special section and the fibre through the cusp, gives the series in $X_e$ for $e \geq 1$ and $k=0$. 

Performing a similar elementary transformation of the curve in $X_1$, this time with center outside the special section, gives in $X_0$ a curve of type $(d,d)$ with one cusp with multiplicity sequence $[d,d-1]$. Since the fibre through the cusp does not intersect the curve outside the cusp, it is again possible to construct the series for $e \geq 0$ and $k=1$. 

It can be shown by induction that a similar construction works for all $e,k \geq 0$, except $(e,k)=(0,0)$; see \cite{MOECCH}. Ultimately, we end up with curves of the given type and Puiseux pairs; see \cite{Brieskorn}.    
 
Note that the second series of curves is constructed from the same plane curves, in this case by blowing up the smooth intersection point of the curve and a generic line through the cusp.
 \end{proof}

\begin{proof}[Proof of Theorem~\ref{thm:finiteness}]
We shall use Theorem~\ref{thm:spectrum} for $x=\frac12$ and study the asymptotic of both sides of \eqref{eq:semic} as $e$ goes to infinity. 
Unfortunately, it turns out that the number of elements of
the spectrum of the singular point of type $(r,s)$ that are contained in $(\frac12,\frac32)$ and the number of elements of the spectrum at infinity that are contained in the same interval, both grow like $\frac34g$, where $g=\frac12(r-1)(s-1)=(a-1)(b-1)+\frac12b(b-1)e$ is the arithmetic genus of the curve. Therefore a
more careful analysis of both terms of inequality \eqref{eq:semic} has to be conducted.

In what follows we will assume that $e$ is large compared to $a,b$. Assume that $C$ is a rational cuspidal curve in $X_e$ of type $(a,b)$ with a single singular point and that this
point has one Puiseux pair $(r,s)$ with $r<s$. We use the notation
\begin{align*}
S_{r,s}&=\#\Sp_{r,s}\cap\left(\frac12,\frac32\right)\\
S_{inf}&=\#\Sp^\infty_{a,b}\cap\left(\frac12,\frac32\right),
\end{align*}
where $\Sp_{r,s}$ is the spectrum of the singular point of type $(r,s)$. In our computation we shall focus on terms linear in $e$, neglecting elements
which are of lower order with respect to $e$. We shall write $x\simeq y$ (where $x$ and $y$ are some expressions depending on $e$) if $\frac{x-y}{e}$ tends
to zero when $e$ goes to infinity.

For example, we have $2g=2(a-1)(b-1)+b(b-1)e\simeq w(b-1)$. Furthermore, by Proposition~\ref{prop:multbound} $r\le b$, hence $(s-1)(r-1)\simeq (r-1)s$.
As $e$ goes to infinity, if $r$ is bounded, $s$ has to go to infinity as well.
\begin{remark}
It might happen that $\frac12$ belongs ot $\Sp^\infty_{a,b}$, so we cannot use Theorem~\ref{thm:spectrum} directly. In that case, we use Theorem~\ref{thm:spectrum}
for a value of $x$ sufficiently close to $\frac12$ (and sufficiently close is a notion depending on $e$: if it is large, $|x-\frac12|$ must be smaller). 
The difference between the number of those elements in the spectrum (we consider $\Sp^\infty_{a,b}$ and $\Sp_{r,s}$)
that are in $(x,1+x)$ and those that are in $(\frac12,\frac32)$ is equal to the multiplicity of $\frac12$ in the spectrum,
so it is $\simeq 0$ as $e$ goes to infinity. As we are interested in the assymptotics only, we will work with the interval $(\frac12,\frac32)$.
\end{remark}

Let us first deal with $S_{r,s}$. Notice, that $S_{r,s}$ is twice the number of elements in $\Sp_{r,s}$ in the interval $(\frac12,1)$. That is
\begin{align*}
\frac12S_{r,s}=&\sum_{i=1}^{r-1}\sum_{j=1}^{s-1}
\begin{cases} 1&\textrm{if }\frac{i}{r}+\frac{j}{s}\in(1/2,1)\\ 0 &\textrm{otherwise}\end{cases}\\
=&\sum_{i=1}^{r-1}\#\{j\colon j\in(\frac{s}{2}-\frac{is}{r},s-\frac{is}{r})\cap\{1,2,\ldots,s-1\}\}\\
=&\sum_{i=1}^{\lfloor r/2\rfloor}\#\{j\in(\frac{s}{2}-\frac{is}{r},s-\frac{is}{r})\cap\{1,2,\ldots,s-1\}\}+\\
+&\sum_{i=\lfloor r/2\rfloor+1}^{r-1}\#\{j\in(0,s-\frac{is}{r})\cap\{1,2,\ldots,s-1\}\}\stackrel{(1)}{\simeq}\\
\simeq&\sum_{i=1}^{\lfloor r/2\rfloor} \frac{s}{2}+\sum_{i=\lfloor r/2\rfloor +1}^{r-1}\left(s-\frac{is}{r}\right)=\\
=&\frac{s}{2}\intfrac{r}{2}+\left(r-1-\intfrac{r}{2}\right)s-\frac{s}{2r}\left(r(r-1)-\intfrac{r}{2}\left(\intfrac{r}{2}+1\right)\right)=\\
=&s\left(\frac{r}{2}-\frac12\intfrac{r}{2}-\frac12+\frac{1}{2r}\intfrac{r}{2}^2+\frac{1}{2r}\intfrac{r}{2}\right).
\end{align*}
The asymptotic equality (1) holds because the difference between the number of integer elements in an interval is equal to its length up to adding or subtracting $1$.
On replacing the number of integers by the length of the interval, the error we make is at most $\pm 1$ at most $r-1$ times, but $r\le b$ is small when compared to $w$.

If $r$ is odd we get $S_{r,s}\simeq 2s\left(\frac{3}{8}r-\frac{1}{4}-\frac{1}{8r}\right)$. If $r$ is even, we get $S_{r,s}=2s\left(\frac{3}{8}r-\frac14\right)$.

\smallskip
We now deal with the spectrum at infinity. Our first observation is that
\[S_{inf}\simeq 2\sum_{p=w/2}^{w-1}\intfrac{pb}{w}.\]
Indeed, in Table~\ref{table:one} we have listed the elements in the spectrum. But the contribution from all the items but the fourth and fifth in the table is bounded
by a function depending only on $a,b$, not on $w$.
We write
\begin{equation}\label{eq:firstsum}
\sum_{p=w/2}^{w-1}\intfrac{pb}{w}\cong\sum_{p=w/2}^w\left(\frac{pb}{w}-\frac12\right)-\sum_{p=w/2}^w\sawtooth{\frac{pb}{w}},
\end{equation}
where $\langle\cdot\rangle$ is the sawtooth function, and ``$\cong$'' means that we neglect the contribution of at most $b$ instances of $p$, where $\frac{pb}{w}$
is an integer.

We have $\sum_{p=w/2}^{w-1} p=\frac{3}{8}w^2+\textrm{lower order terms in $w$}$. 
Therefore the first sum in \eqref{eq:firstsum} gives a value asymptotically equal to $w(\frac38b-\frac14)$. As for the second sum,
in \eqref{eq:firstsum} we shall use the following lemma.
\begin{lemma}\label{lem:elementarynumbertheory}
For fixed $b>1$ we have
\[\lim_{w\to\infty}\frac{1}{w}\sum_{p=w/2}^{w-1}\sawtooth{\frac{pb}{w}}=
\begin{cases}
0& \textrm{$b$ is even,}\\
\frac{1}{8b} & \textrm{$b$ is odd.}
\end{cases}\]
\end{lemma}

The proof of Lemma~\ref{lem:elementarynumbertheory} is postponed until Section~\ref{proofofelementarynumbertheory}.
Now we resume the proof of Theorem~\ref{thm:finiteness}. 
We observe that 
\begin{equation}\label{eq:similargenus}
s(r-1)\simeq 2g\simeq w(b-1).
\end{equation}
Our computations insofar show that
\[S_{r,s}\simeq 
\begin{cases} \frac34g+\frac14s & \textrm{$r$ is even,}\\
\frac34g+\left(\frac14-\frac{1}{4r}\right)s &\textrm{$r$ is odd.}
\end{cases}\]
As for $S_{inf}$ we get
\[S_{inf}\simeq
\begin{cases} \frac34g+\frac14w & \textrm{$b$ is even,}\\
\frac34g+\left(\frac14-\frac{1}{4b}\right)w & \textrm{$b$ is odd.}
\end{cases}\]

Depending on the parity of $b$ and $r$ we have four cases. 
\begin{itemize}
\item $b$ and $r$ are even. Then $S_{r,s}$ is asymptotically bigger than $S_{new}$ unless $s\simeq w$. But $s(r-1)\simeq w(b-1)$, so this exceptional
case can occur only if $b=r$. Theorem~\ref{thm:spectrum} obstructs asymptotically all cases with $r<b$.
\item $b$ and $r$ odd. We compare $w\frac{b-1}{b}$ with $s\frac{r-1}{r}$. Since $s(r-1)\simeq w(b-1)$, we see that for large $e$ the quantity $S_{inf}$ is smaller
than $S_{r,s}$ unless $b=r$.
\item $b$ odd and $r$ even. The case is completely obstructed (for $e$ large) by Theorem~\ref{thm:spectrum}.
\item $b$ even and $r$ odd. We compare $s\frac{r-1}{r}$ with $w\frac{b-1}{b-1}$. If $r<b-1$ the
we have asymptotically $S_{r,s}>S_{inf}$, so Theorem~\ref{thm:spectrum} applies. We cannot obstruct the case $r=b-1$.
\end{itemize}
\end{proof}

\subsection{Proof of Lemma~\ref{lem:elementarynumbertheory}}\label{proofofelementarynumbertheory}
In this section we use the notions of Dedekind sums. For the reader's convenience we provide a short summary in Section~\ref{s:dedekind}.

Our goal is to compute the limit of the sequence $\frac{1}{w}a_w$, where 
\[a_w:=\sum_{p=w/2}^{w-1}\sawtooth{\frac{pb}{w}}.\]

We write
\begin{equation}\label{eq:split}
\begin{split}
a_w=&\sum_{p=w/2}^{w-1}\sawtooth{\frac{pb}{w}}=\sum_{p=w/2}^{w-1}\sawtooth{\frac{pb}{w}}\intfrac{2p}{w}=\\
&=\sum_{p=0}^{w-1}\sawtooth{\frac{pb}{w}}\intfrac{2p}{w}=
\sum_{p=0}^{w-1}\sawtooth{\frac{pb}{w}}\left(\frac{2p}{w}-\sawtooth{\frac{2p}{w}}+\frac12\right)=\\
&=\sum_{p=0}^{w-1}\sawtooth{\frac{pb}{w}}\frac{2p}{w}-\sum_{p=0}^{w-1}\sawtooth{\frac{pb}{w}}\sawtooth{
\frac{2p}{w}}+\frac12\sum_{p=0}^{w-1}\sawtooth{\frac{pb}{w}}.
\end{split}
\end{equation}
We write
\begin{align*}
b_w&=\sum_{p=0}^{w-1}\sawtooth{\frac{pb}{w}}\frac{2p}{w}\\
c_w&=\sum_{p=0}^{w-1}\sawtooth{\frac{pb}{w}}\sawtooth{\frac{2p}{w}}\\
d_w&=\frac12\sum_{p=0}^{w-1}\sawtooth{\frac{pb}{w}}.
\end{align*}
so that $a_w=b_w-c_w+d_w$. The computation of the limit $\frac1wa_w$ splits into three lemmas of increasing difficulty.

\begin{lemma}\label{lem:dm}
We have $d_w=0$.
\end{lemma}
\begin{proof}[Proof of Lemma~\ref{lem:dm}]
Suppose $b$ and $w$ are coprime. Then $b$ is invertible modulo $w$, so after changing variables the sum becomes 
$\sum_{p=0}^{w-1}\sawtooth{\frac{p}{w}}$, which is zero by elementary calculations. If $b$ and $w$ are not coprime,
we write $c=\gcd(b,w)$, $w'=w/c$, $b'=b/c$, and the sum is equal to $c\sum_{p=0}^{w'-1}\sawtooth{\frac{pb'}{w'}}$, so
we reduce to the previous case.
\end{proof}

\begin{lemma}\label{lem:bm}
We have $\lim\dfrac1wb_w=\dfrac{1}{6b}$.
\end{lemma}
\begin{proof}[Proof of Lemma~\ref{lem:bm}]
Let again $c=\gcd(b,w)$ and $w'=w/c$, $b'=b/c$. We have
\[b_w=\frac{2c}{b}\sum_{p=0}^{w'-1}\sawtooth{\frac{pb'}{w'}}\frac{pb'}{w'}.\]
The same argument as in Lemma~\ref{lem:dm} allows us to rewrite the sum in the following way.
\[\sum_{p=0}^{w'-1}\sawtooth{\frac{pb'}{w'}}\frac{pb'}{w'}=\sum_{p=0}^{w'-1}\sawtooth{\frac{pb'}{w'}}^2.\]
By assumption, $b'$ and $w'$ are coprime. Substituting for $p$ the multiple $pb''$, where $b''$ is the inverse of $b'$ modulo $w'$
we obtain 
\[\sum_{p=0}^{w'-1}\sawtooth{\frac{pb'}{w'}}^2=\sum_{p=0}^{w'-1}\sawtooth{\frac{p}{w'}}^2=s(1,w'),\]
By Theorem~\ref{thm:dedekind} (and the elementary observation that $s(w',1)=0$) the expression 
evaluates to $\frac{1}{12}\left(w'+\frac{2}{w'}-3\right)$.
The lemma follows immediately.
\end{proof}

\begin{lemma}\label{lem:cm}
If $b$ is odd, we have $\lim\dfrac1wc_w=\dfrac{1}{24b}$, while if $b$ is even, then $\lim\dfrac1wc_w=\dfrac{1}{6b}$.
\end{lemma}
\begin{proof}
The term $c_w$ is the Rademacher--Dedekind symbol $D(2,b,w)$. We would like to use the Rademacher reciprocity law, but in order to do this,
we need to pass to a subsequence of $w$, because we need that $2,b,w$ are pairwise coprime. 
The following claim allows us to compute $\lim\frac1wc_w$ by passing to a subsequence 
given by some arithmetic progression.

\emph{Claim.} $|c_w-c_{w+1}|<\frac32b+\frac{33}{4}.$

To prove the claim notice that $\delta_{p,b,w}:=\sawtooth{\frac{pb}{w}}-\sawtooth{\frac{pb}{w+1}}$
is equal to $\frac{pb}{w(w+1)}<\frac{b}{w}$, unless any one of the three conditions holds:
\begin{itemize}
\item $\intfrac{pb}{w}>\intfrac{pb}{w+1}$;
\item $\intfrac{pb}{w}$ is an integer;
\item $\intfrac{pb}{w+1}$ is an integer;
\end{itemize}
In each of these cases we have $\delta_{p,b,w}<1$. The last two cases occur at most $b$ times each. If the first case occurs, then
there exists $k\in\Z$ such that $\frac{pb}{w}\ge k>\frac{pb}{w+1}$. This
implies that $p\in[\frac{wk}{b},\frac{(w+1)k}{b})$. Notice that $p<w$, hence $k<b$ and the length of the interval is at most $1$. Therefore
it cannot contain more than one integer. Hence the first case occurs at most $b$ times. Combining these cases (they are
not mutually exclusive, but we can be slightly wasteful) we obtain
\begin{equation}\label{eq:deltapm}
|\delta_{p,b,w}|\le\begin{cases} 1 & \textrm{ for at most $3b$ values of $p\in\{0,1,\ldots,w-1\}$}\\
\frac{1}{w} & \textrm{ for all other values of $p$.}\end{cases}
\end{equation}

We now write
\begin{multline*}
\left|\sawtooth{\frac{pb}{w}}\sawtooth{\frac{2p}{w}}-\sawtooth{\frac{pb}{w+1}}\sawtooth{\frac{2p}{w+1}}\right|\le\\
\left|\sawtooth{\frac{pb}{w}}-\sawtooth{\frac{pb}{w+1}}\right| \left|\sawtooth{\frac{2b}{w}}\right|+
\left|\sawtooth{\frac{2b}{w+1}}-\sawtooth{\frac{2b}{w}}\right| \left|\sawtooth{\frac{pb}{w+1}}\right|\le\\
\le \frac12|\delta_{p,b,w}|+\frac12|\delta_{p,2,w}|,
\end{multline*}
where we used the fact that $|\sawtooth{x}|\le\frac12$ for all $x$. We sum up the above inequality over $p\in\{0,\ldots,w-1\}$.
Combining this with \eqref{eq:deltapm} and adding a $\frac14$ for the term $p=w$, which appear in the formula for $c_{w+1}$ (and do
not appear in the formula for $c_w$), we obtain

\[
|c_w-c_{w+1}|
\le\frac14+\frac12\sum_{p=0}^{w-1}\left|\delta_{p,b,w}\right|+\left|\delta_{p,2,w}\right|
\le \frac14+\frac12(3b+1)+\frac12(3\cdot 2+1)=\frac32b+\frac{33}{4}.
\]
This proves the claim. We now resume the proof of Theorem~\ref{thm:finiteness}. We split it into two cases.

\smallskip
\underline{Case 1. $b$ is odd.}
Suppose $w$ is coprime to $b$ and $2$. By the Rademacher reciprocity law (Theorem~\ref{thm:rademacher}):
\[c_w=\frac{b^2+w^2+4-6bw}{24bw}-\sum_{p=0}^2\sawtooth{\frac{pw}{2}}\sawtooth{\frac{bw}{2}}-\sum_{p=0}^{b-1}\sawtooth{\frac{2p}{b}}\sawtooth{\frac{pw}{b}}.\]
The two sums on the left are sums of bounded (as $w$ goes to infinity) number of summands with each summand bounded by $\frac14$. It follows that
\[\lim_{\substack{w\to\infty\\ \textrm{$w$ coprime with $2$ and $b$}}} \frac1wc_w=\lim_{w\to\infty}\frac{1}{w}\frac{b^2+w^2+4-8bw}{24bw}=\frac{1}{24b}.\]
By the claim, the limit of the subsequence of $\frac1wc_w$ over $w$ coprime with $2$ and $b$ is the same as the limit of the sequence 
$\frac1wc_w$.

\smallskip
\underline{Case 2. $b$ is even.}
Suppose $w$ is even and $\gcd(b,w)=2$. Write $w'=w/2$, $b'=b/2$. Then we have
\[c_w=2\sum_{p=0}^{w'-1}\sawtooth{\frac{pb'}{w'}}\sawtooth{\frac{p}{w'}}=2s(b',w').\]
By the Dedekind reciprocity law (Theorem~\ref{thm:dedekind}):
\[s(b',w')=\frac{1}{12}\left(\frac{w'}{b'}+\frac{w'}{b'}+\frac{1}{b'w'}-3\right)-s(w',b').\]
The expression $s(w',b')$ is a sum of $b'$ terms, each bounded by $\frac14$, hence $\lim_{w'\to\infty}\frac{1}{w'}s(w',b')=0$. This means that
\[\lim_{\substack{w\to\infty\\\gcd(b,w)=2}}\frac{1}{w}c_w=\frac{1}{6b}.\]
By the claim, the limit of a subsequence of $\frac1wc_w$ is the same as the limit of the full sequence.
\end{proof}

\subsection{Dedekind sums and reciprocity laws}\label{s:dedekind}

For the reader's convenience we include some elementary facts about Dedekind sums. We refer to \cite{RG} or \cite{HZ} for more detail.

Let $p,q$ be positive integers. The \emph{Dedekind sum} is the following expression.
\[ s(p,q)=\sum_{i=0}^{q-1}\sawtooth{\frac{i}{q}}\sawtooth{\frac{pi}{q}}.\]

Dedekind sums appear in number theory in various places; see for instance \cite{RG}; in combinatorics, when they are used to compute
the number of lattice points in polytopes; and in low-dimensional topology where they appear in connection with lens spaces, Seifert fibred manifolds
and signatures of torus knots: see \cite{HZ} for more detail.

The most important result we use is the Dedekind reciprocity law.
\begin{theorem}\label{thm:dedekind}
If $p$ and $q$ are coprime, then
\[s(p,q)+s(q,p)=\frac{1}{12}\left(\frac{p}{q}+\frac{q}{p}+\frac{1}{pq}-3\right).\]
\end{theorem}

There are many ways to generalise the Dedekind sum, of which we use only one, the so--called Dedekind--Rademacher sum $D(p,q,r)$ defined as
\[D(p,q,r)=\sum_{i=0}^{r-1}\sawtooth{\frac{pi}{r}}\sawtooth{\frac{qi}{r}}.\]
The Dedekind reciprocity law generalises to the Rademacher reciprocity law (or the ``three term law'') which is stated as follows.
\begin{theorem}\label{thm:rademacher}
If $p,q,r$ are pairwise coprime, then
\[D(p,q,r)+D(r,p,q)+D(q,r,p)=\frac{p^2+q^2+r^2-3pqr}{12pqr}.\]
\end{theorem}

\bibliographystyle{amsplain}
\def\MR#1{}
\bibliography{bib.bib}

\end{document}